\RequirePackage{fix-cm}
\documentclass[smallextended]{svjour3}       

\usepackage{lipsum}
\usepackage{amsfonts}
\usepackage{graphicx}
\usepackage{epstopdf}
\usepackage{braket,amsfonts}
\usepackage{amsmath,amssymb}
\usepackage{stmaryrd}
\usepackage{array}
\usepackage{threeparttable, booktabs, caption, makecell, multirow}
\usepackage[caption=false]{subfig}
\usepackage{algorithmic, algorithm}
\usepackage{graphicx,epstopdf}

\usepackage{pgfplots}
\pgfplotsset{compat=newest}

\usepackage{xspace}
\usepackage{bold-extra}
\usepackage{tcolorbox}

\ifpdf
  \DeclareGraphicsExtensions{.eps,.pdf,.png,.jpg}
\else
  \DeclareGraphicsExtensions{.eps}
\fi

\usepackage{amsopn}

\DeclareMathOperator*{\argmin}{argmin}

\newtheorem{assumption}{Assumption}

\newcommand{\image}{\mathrm{Im}}

\newcommand{\Z}{\mathbb{Z}}
\newcommand{\R}{\mathbb{R}}
\colorlet{texcscolor}{blue!50!black}
\colorlet{texemcolor}{red!70!black}
\colorlet{texpreamble}{red!70!black}
\colorlet{codebackground}{black!25!white!25}

\newcommand\clearrow{\global\let\rowmac\relax}	
\clearrow

\newcommand{\rev}[1]{\textcolor{black}{#1}}
\smartqed  
\usepackage{graphicx}
\begin{document}

\title{A Two-Level Distributed Algorithm for Nonconvex Constrained Optimization 
}

\author{Kaizhao Sun \and
		X. Andy Sun 
}

\institute{Kaizhao Sun \at
			H. Milton Stewart School of Industrial and Systems Engineering, Georgia Institute of Technology, Atlanta, GA \\
              \email{ksun46@gatech.edu}        
           \and
           X. Andy Sun \at
           Sloan School of Management, Massachusetts Institute of Technology, Cambridge, MA \\
              \email{sunx@mit.edu}
}

\date{Received: date / Accepted: date}

\maketitle
\begin{abstract}
This paper aims to develop distributed algorithms for nonconvex optimization problems with complicated constraints associated with a network. The network can be a physical one, such as an electric power network, where the constraints are nonlinear power flow equations, or an abstract one that represents constraint couplings between decision variables of different agents. Despite the recent development of distributed algorithms for nonconvex programs, highly complicated constraints still pose a significant challenge in theory and practice. 
We first identify some difficulties with the existing algorithms based on the alternating direction method of multipliers (ADMM) for dealing with such problems. 
We then propose a reformulation that enables us to design a two-level algorithm, which embeds a specially structured three-block ADMM at the inner level in an augmented Lagrangian method (ALM) framework.
Furthermore, we prove the global and local convergence as well as iteration complexity of this new scheme for general nonconvex constrained programs, and show that our analysis can be extended to handle more complicated multi-block inner-level problems.
Finally, we demonstrate with computation that the new scheme provides convergent and parallelizable algorithms for various nonconvex applications, and is able to complement the performance of the state-of-the-art distributed algorithms in practice by achieving either faster convergence in optimality gap or in feasibility or both.
\keywords{Distributed Optimization \and Augmented Lagrangian Method \and  Alternating Direction Method of Multipliers}
\subclass{90C06 \and 90C26 \and 90C30 \and 90C35}

\end{abstract}

\section{Introduction}\label{section: Introduction}
This paper develops a new two-level distributed algorithm with global and local convergence guarantees for solving general smooth and nonsmooth {constrained} nonconvex optimization problems. We will start with a general constrained optimization model that is motivated by nonlinear network flow problems, and then explore reformulations for distributed computation. This process of reformulation leads us to observe two structural properties that a distributed reformulation should possess, which in fact pose some challenges to existing distributed algorithms in terms of convergence and practical performance. This observation inspired us to develop the two-level distributed algorithm. We will summarize our contributions in the end of this section.

\subsection{Constrained Nonconvex Optimization over a Network}
Consider a connected, undirected graph\footnote{In this paper, we use ``networks'' and ``graphs'' interchangeably.} $G(\mathcal{V},\mathcal{E})$ with a set of nodes $\mathcal{V}$ and a set of edges $\mathcal{E}$. A centralized constrained optimization problem on $G$ is given as
\begin{subequations}\label{eq:networkExample}
	\begin{align}
		\min \quad& \sum_{i\in \mathcal{V}}f_i(x_i) \label{eq:networkExampleObj}\\
		\mathrm{s.t.} \quad 
		& h_i(x_i,\{x_j\}_{j\in\delta(i)}) = 0, \quad \forall i \in \mathcal{V}, \label{eq:networkExampleC1} \\
		& g_i(x_i, \{x_j\}_{j\in \delta(i)}) \le 0,\quad \forall i \in \mathcal{V},\label{eq:networkExampleC2}\\
		& x_i\in \mathcal{X}_i, \quad \forall i \in \mathcal{V}, \label{eq:networkExampleC3}
	\end{align}
\end{subequations}
where each node $i\in\mathcal{V}$ of the graph $G$ is associated with a decision variable $x_i$ and a cost function $f_i(x_i)$ as in \eqref{eq:networkExampleObj}. Variable $x_i$ and variables $x_j$ of $i$'s adjacent nodes $j\in\delta(i)$  are coupled through constraints \eqref{eq:networkExampleC1}-\eqref{eq:networkExampleC2}, and $\mathcal{X}_i$  in \eqref{eq:networkExampleC3} represents some constraints only for $x_i$. The functions $f_i$, $h_i$, $g_i$, and the set $\mathcal{X}_i$ may be nonconvex.

Any constrained optimization problem can be reformulated as \eqref{eq:networkExample} after proper transformation. An especially interesting motivation for us is the nonlinear network flow problems. In this case, the graph $G$ represents a physical network such as an electric power network, a natural gas pipeline network, or a water transport network, where the variables $x_i$ in \eqref{eq:networkExample} are nodal potentials such as electric voltages, gas pressures, or water pressures, and the constraints $h_i$ and $g_i$ are usually nonconvex functions that describe the physical relations between nodal potentials and flows on the edges, flow balance at nodes, and flow capacity constraints. Notice that a node $i$ in the graph can also represent a sub-network of the entire physical network, and then the constraints could involve variables in adjacent sub-networks. There has been much recent interest in solving nonlinear network flow problems, e.g. applications in the optimal power flow problem in electric power network \cite{kocuk2016strong}, the natural gas nomination problem \cite{pfetsch2015validation}, and the water network scheduling problem \cite{d2015mathematical}. 

In many situations, it is desirable to solve problem \eqref{eq:networkExample} in a distributed manner, where each node $i$ represents an individual agent that solves a localized problem, while agents coordinate with their neighbors to solve the overall problem. Each agent need to handle its own set of local constraints $h_i$, $g_i$, and $\mathcal{X}_i$. For example, agents may be geographically dispersed with local constraints representing the physics of the subsystems, which cannot be controlled by other agents; or agents may have private data in their constraints, which cannot be shared with other agents; or the sheer amount of data needed to describe constraints or objective could be too large to be stored or transmitted in distributed computation between agents. These practical considerations pose restrictions that each agent in a distributed algorithm has to deal with a set of complicated, potentially nonconvex, constraints. 

\subsection{Necessary Structures of Distributed Formulations}\label{subsub: property}
In order to do distributed computation, the centralized formulation \eqref{eq:networkExample} first needs to be transformed into a formulation to which a distributed algorithm could be applied. We call such a formulation a \emph{distributed formulation}, whose form may depend on specific distributed algorithms as well as on the structure of the distributed computation, e.g. which variables and constraints are controlled by which agents and in what order computation and communication can be carried out. Despite the great variety of distributed formulations, we want to identify some desirable and necessary features for a distributed formulation. 

One desirable feature is the capability of \textit{parallel decomposition} so that all agents can solve their local problems in parallel, rather than in sequence. To realize this, each agent needs a local copy of its neighboring agents' variables. For problem \eqref{eq:networkExample}, we may introduce a local copy $x^i_j$ of the original variable $x_j$ and a global copy $\bar{x}_j$, and enforce consensus as
	\begin{align}\label{eq:dupSchemeParallel}
		x_j = \bar{x}_j, \; x^i_j = \bar{x}_j, \quad\forall j\in\mathcal{V}, \; i\in\delta(j).
	\end{align} 
	Using this duplication scheme, a distributed formulation of \eqref{eq:networkExample} can be written as 
	\begin{subequations}\label{formulation: generic} 
		\begin{align}
		\min_{x, \bar{x}} \quad & f({x})=\sum_{i\in\mathcal{V}} f_i(x^i) \label{eq: DistributedObj0} \\
		\mathrm{s.t.}  \quad  
		& A{x}+B\bar{x}=0\label{eq:DistributedConcensus0},\\
		& x^i \in \mathcal{X}_i,~~\forall i\in \mathcal{V},\quad \color{black}{\bar{x}\in \bar{\mathcal{X}}}.\label{eq:DistributedConstr}
	\end{align}
	\end{subequations}	
\textcolor{black}{
    In problem \eqref{formulation: generic}, the optimization variables are ${x}=[\{x^i\}_{i\in\mathcal{V}}]\in \R^{n_1}$ and $\bar{x}=[\{\bar{x}_j\}_{j\in\mathcal{V}}]\in \R^{n_2}$. Each subvector $x^i=[x_i,\{x^i_j\}_{j\in\delta(i)}]\in \R^{n_{1i}}$ of $x$ denotes all the local variables controlled by agent $i$ including the original variable $x_i$ and the local copies $x^i_j$; each subvector $\bar{x}_j$ of $\bar{x}$ denotes a global copy of $x_j$. The set $\mathcal{X}_i\subseteq \R^{n_{1i}}$ is defined as $\mathcal{X}_i:=\{v:~h_i(v) = 0,~g_i(v)\leq 0\}$, so the original constraints \eqref{eq:networkExampleC1}-\eqref{eq:networkExampleC2} are decoupled into each agent's local constraints $\mathcal{X}_i$, which also absorb the constraints \eqref{eq:networkExampleC3}. Additionally, the global copy $\bar{x}$ is constrained in some simple convex set $\bar{\mathcal{X}}\subseteq \R^{n_2}$. The only coupling among agents are  \eqref{eq:DistributedConcensus0}, which formulate the consensus constraint \eqref{eq:dupSchemeParallel} with $A\in \R^{m\times n_1}$ and $B\in \R^{m\times n_2}$.} 
An alternating optimization scheme is then natural, as all the agents can solve their subproblems over $x^i$'s in parallel once $\bar{x}$ is fixed; and once $x^i$'s are updated and fixed, the subproblems over $\bar{x}$ can also be solved in parallel. 

In fact, for any constrained optimization problem, not necessarily a network flow type problem, if distributed computation is considered, variables of the centralized problem need to be grouped into variables $x^i$ in a distributed formulation for agents $i$ according to the decision structure, and duplicate variables $\bar{x}$ need to be introduced to decouple the constraints from agents. In this way, problem \eqref{formulation: generic} provides a general formulation for distributed computation of constrained optimization problems. Conversely, due to the necessity of duplicating variables, any distributed formulation of a constrained program \emph{necessarily} shares some key structures of \eqref{formulation: generic}. In particular, problem \eqref{formulation: generic} has two simple but crucial properties. Namely, 
\begin{itemize}
	\item Property 1: As the matrices $A$ and $B$ are defined in \eqref{eq:dupSchemeParallel}, the image of $A$ strictly contains the image of $B$, i.e. $\image(A) \supsetneq \image(B)$. 
	\item Property 2: Each agent $i$ may face local nonconvex constraints $\mathcal{X}_i$.
\end{itemize}

Property 1 follows from the fact that, for any given value of $\bar{x}_j$ in \eqref{eq:dupSchemeParallel}, there is always a feasible solution $(x_j, x^i_j)$ that satisfies the equalities in \eqref{eq:dupSchemeParallel}, but if $x_j\ne x^i_j$, then there does not exist an $\bar{x}_j$ that satisfies both equalities in \eqref{eq:dupSchemeParallel}.  Property 2 follows from our desire to decompose the computation for different agents.

	
In this paper, we will show that the above two properties of distributed constrained optimization pose a significant challenge to the theory and practice of existing distributed optimization algorithms. In particular, existing distributed algorithms based on the alternating direction method of multipliers (ADMM) may fail to converge for the general nonconvex constrained problem \eqref{formulation: generic} without further reformulation or relaxation. 
Before proceeding, we summarize our contributions.

\subsection{Summary of Contributions}
The contributions of the paper can be summarized below.

Firstly, we propose a new reformulation and a two-level distributed algorithm for solving nonconvex constrained optimization problem \eqref{eq:networkExample}-\eqref{eq:dupSchemeParallel}, 
which embeds a specially structured three-block ADMM at the inner level in an augmented Lagrangian method (ALM) framework. The proposed algorithm maintains the flexibility of ADMM in achieving distributed computation.

Secondly, we prove global and local convergence as well as iteration complexity results for the proposed two-level algorithm, and illustrate that the underlying algorithmic framework can be extended to more complicated nonconvex multi-block problems. {For the convergence of ADMM, we allow each nonconvex subproblem to be solved to a stationary point with certain improvement in the objective function compared to the previous iterate, which mildly relaxes the global optimality of nonconvex subproblems commonly assumed in the ADMM literature. Our convergence analysis builds on the classical and recent works on ADMM and ALM, and our results are derived by relating these two methods in an analytical way.}

Thirdly, we provide extensive computational tests of our two-level algorithm on nonconvex network flow problems, parallel minimization of nonconvex functions over compact manifolds, and a robust tensor PCA problem from machine learning. Numerical results demonstrate the advantages of the proposed algorithm over existing ones, including randomized updates, modified ADMM, and centralized solver, either in the convergence speed to close optimality gap, or to close feasibility gap, or both. Moreover, our test result on the multi-block robust tensor PCA problem suggests that the proposed two-level algorithm not only ensures convergence for a wider range of applications where ADMM may fail, but also tends to accelerate ADMM on problems where convergence of ADMM is already guaranteed.

\subsection{Notation}
Throughout this paper, we use $\Z_{+}$ (resp. $\Z_{++}$) to denote the set of nonnegative (resp. positive) integers, and $\R^n$ to denote the $n$-dimensional real Euclidean space. For $x,y\in \R^n$, the inner product is denoted by $x^\top y$ or $\langle x, y \rangle$; the Euclidean norm is denoted by $\|x\| = \sqrt{\langle x, x\rangle}$. A vector $x$ may consist of $J$ subvectors $x_j\in \R^{n_j}$ with $\sum_{j=1}^J n_j=n$; in this case, we will write $x = [\{x_j\}_{j\in [J]}]$, where $[J] = \{1,\cdots, J\}$. Occasionally, we use $x_i$ to denote the $i$-th component of $x$ if there is no confusion to do so. 
\rev{For a matrix $A$, denote its largest singular value by $\|A\|$ and image space by $\mathrm{Im}(A)$.}
We use $B_r(x)$ to denote the Euclidean ball centered at $x$ with radius $r>0$.
For a closed set $C\subset \R^n$, the interior of $C$ is denoted by $\mathbf{Int}~C$, the projection operator onto $C$ is denoted by $\mathrm{Proj}_C(x)$, and the indicator function of $C$ is denoted by $\mathbb{I}_C(x)$, which takes value 0 if $x\in C$ and $+\infty$ otherwise. 

The rest of this paper is organized as follows. In Section \ref{sec:literature}, we review the literature and summarize two conditions that are crucial to the convergence of ADMM, which are essentially contradicting to Properties 1 and 2. In Section \ref{sec:Reformulation}, we propose our new reformulation and a two-level algorithm for solving problem \eqref{formulation: generic} in a distributed way. 
In Section \ref{section: Convergence}, we provide the global convergence as well as the iteration complexity result, and show our scheme can be applied to more complicated multi-block problems. Then in Section \ref{section: Local}, we show the local convergence result under standard second-order assumptions. Finally, we present computational results in Section \ref{section: Examples} and conclude in Section \ref{section: Conclusion}.

\section{Related Literature} \label{sec:literature}
In this section, we review the literature on ADMM and other distributed algorithms, and identify some limitations of the standard ADMM approach in solving problem \eqref{formulation: generic}.

\subsection{Earlier Works and ADMM for Convex Problems}
ALM and the method of multipliers (MoM) were proposed in the late 1960s by Hestenes \cite{hestenes1969multiplier} and Powell \cite{powell1967method}. ALM enjoys more robust convergence properties than dual decomposition \cite{bertsekas1973convergence,rockafellar1973multiplier}, and convergence for partial elimination of constraints has been studied \cite{bertsekas2014constrained}. ADMM was proposed by Glowinski and Marrocco \cite{glowinski1975approximation} and Gabay and Mercier \cite{gabay1976dual} in the mid-1970s, and has deep roots in maximal monotone operator theory and numerical algorithms for solving partial differential equations \cite{eckstein1992douglas,douglas1956numerical,peaceman1955numerical}. ADMM solves the subproblems in ALM by alternately optimizing through blocks of variables and in this way achieves distributed computation. 
The convergence of ADMM with two block variables is proved for convex optimization problems \cite{glowinski1975approximation,gabay1976dual,gabay2024applications,eckstein1992douglas} and the $\mathcal{O}(1/k)$ convergence rate is established \cite{he20121,monteiro2013iteration,he2015non}. 
Some applications in distributed consensus problems include \cite{boyd2011distributed,wei_distributed_2012,shi_linear_2014,makhdoumi_broadcast-based_2014,aus_ozdaglar_distributed_2015,makhdoumi_convergence_2016}. More recent convergence results on multi-block convex ADMM can be found in \cite{he_alternating_2012,he2012convergence,han2012note,chen2013convergence,lin2015sublinear,li2015convergent,lin2015global,lin2016iteration,chen2016direct,hong2017linear,davis2017three,lin2018global}.

\subsection{ADMM for Nonconvex Problems}
The convergence of ADMM has been observed for many nonconvex problems with various applications in matrix completion and factorization \cite{xu2012alternating,shen2014augmented,zhang2014asynchronous,yang2015alternating}, optimal power flow \cite{sun2013fully,erseghe2014distributed,magnusson2015distributed}, asset allocation \cite{wen2013asset}, and polynomial optimization \cite{jiang2014alternating}, among others. 
For convergence theory, several conditions have been proposed to guarantee convergence on structured nonconvex problems that can be abstracted in the following form
	\begin{align}\label{eq:genericADMM}
		\min_{x_1,\cdots, x_p, z}~&\sum_{i=1}^p f_i(x_i) + h(z) + g(x_1,\cdots,x_p,z)\\
		~\mathrm{s.t.}~&\sum_{i=1}^pA_ix_i + Bz = b,~ x_i\in \mathcal{X}_i~\forall i\in [p]. \notag
	\end{align}
We summarize some convergence conditions in Table \ref{table:nonconvexADMM}. For instance, Hong et al. \cite{hong2016convergence} studied ADMM for nonconvex consensus and sharing problems under cyclic or randomized update order. Li and Pong \cite{li2015global} and Guo et al. \cite{guo2017convergence} studied two-block ADMM, where one of the blocks is the identity matrix. One of the most general frameworks for proving convergence of multi-block ADMM is proposed by Wang et al. \cite{wang2015global}, where the authors showed a global subsequential convergence with a rate of $o(1/\sqrt{k})$. A more recent work by Themelis and Patrinos \cite{themelis2018douglas} established a primal equivalence of nonconvex ADMM and Douglas-Rachford splitting.
    
Another line of research explores some variants of ADMM. Wang et al. \cite{wang2014convergence,wang2015convergence} studied the nonconvex Bregman-ADMM, where a Bregman divergence term is added to the augmented Lagrangian function during each block update to facilitate the descent of certain potential function. Gon{\c{c}}alves, Melo, and Monteiro \cite{gonccalves2016extending} provided an alternative convergence rate proof of proximal ADMM applied to convex problems, which was shown to be an instance of a 
more general non-Euclidean hybrid proximal extragradient framework. The two-block, multi-block, and Jacobi-type extensions of this framework to nonconvex problems can be found in \cite{gonccalves2017convergence,melo2017iteration,melo2017jacobi}, where an iteration complexity of $\mathcal{O}(1/\sqrt{k})$ was also established. Jiang et al. \cite{jiang2019structured} proposed two variants of proximal ADMM. Some proximal terms are added to the first $p$ block updates; for the last block, either a gradient step is performed, or a quadratic approximation of the augmented Lagrangian is minimized. 
	
\begin{table}[h!]
	{\footnotesize
		\caption{Comparisons of the Nonconvex ADMM Literature}\label{table:nonconvexADMM}
	\begin{center}
	\begin{tabular}{|c|c|c|c|c|c|c|c|}
		\hline
		 	  & $p$ & $f_i$'s & $\mathcal{X}_i$'s & $h$ & $g$ & $A_i$'s & $B$ \\
		\hline
		\multirowcell{3}{\cite{hong2016convergence}} & 1        & convex  & convex  & smooth   & -&  -& $I$   \\
		\cline{2-8}
				& \multirowcell{2}{$\geq 2$} & convex  & \multirowcell{2}{convex} & \multirowcell{2}{smooth}   & \multirowcell{2}{-}&  \multirowcell{2}{full col.}& \multirowcell{2}{$I$}  \\
		\cline{3-3}
								   &  & smooth nonconvex  &  &    & & &   \\

		\hline
		\cite{li2015global}       & 1       &\multicolumn{2}{c|}{l.s.c}  &$\nabla^2h$ bounded &-& $I$ & full row   \\
		\hline   
		\cite{guo2017convergence} & 1       &\multicolumn{2}{c|}{l.s.c}  & smooth & - & full col. & $I$ \\
		\hline    
		\cite{wang2014convergence}     & 1      &\multicolumn{2}{c|}{l.s.c \& $f_1+h$ subanalytic}   & smooth & - & full col. & full row \\
		\hline    
		\cite{wang2015convergence}     & 2      &\multicolumn{2}{c|}{l.s.c \&  $f_1+f_2+h$ subanalytic}   & smooth & - & - & full row \\
		\hline 
		\multirowcell{2}{\cite{wang2015global}}     &\multirowcell{2}{$\geq2$}  &\multicolumn{2}{c|}{l.s.c \& restricted prox-regular} &\multicolumn{2}{c|}{\multirowcell{2}{smooth}} &\multicolumn{2}{c|}{$\image([A,b]) \subseteq \image(B)$} \\
		\cline{3-4}
			    & &\multicolumn{2}{c|}{$\partial f_1$ bounded \& $f_{>1}$'s p.w. linear } &\multicolumn{2}{c|}{} &\multicolumn{2}{c|}{Lip. sub-min path} \\
		\hline
		\cite{gonccalves2017convergence,melo2017iteration}   &$1,2$ &\multicolumn{2}{c|}{l.s.c} &  $\approx$smooth & - &\multicolumn{2}{c|}{$\image([A,b]) \subseteq \image(B)$}\\
		\hline
		\cite{melo2017jacobi}   &$\geq 2$ &\multicolumn{2}{c|}{l.s.c} & smooth & - &\multicolumn{2}{c|}{$\image([A,b]) \subseteq \image(B)$}\\
		\hline
		\multirowcell{2}{\cite{jiang2019structured}} &\multirowcell{2}{$\geq2$} & Lipschitz continuous & compact &\multicolumn{2}{c|}{\multirowcell{2}{smooth}} & \multirowcell{2}{-} & $I$ \textbf{or} \\ 
		\cline{3-4}
			&  &\multicolumn{2}{c|}{l.s.c}   &\multicolumn{2}{c|}{} &  & full row\\
		\hline
		\multicolumn{8}{c}{l.s.c: lower semi-continuous; smooth: Lipschitz differentiable; full col./row: full column/row rank}\\
	\end{tabular}
	\end{center}}	
\end{table}

For general nonconvex and \rev{nonsmooth} problems, we note that the convergence of ADMM relies on the following two conditions.
\begin{itemize}
	\item Condition 1: Denote $A:=[A_0, \cdots, A_p]$, then 
		$\image([A,b])$ $\subseteq \image(B)$.
	\item Condition 2: The last block objective function $h(z)$ is Lipschitz differentiable.
\end{itemize}
Due to the sequential update order of ADMM, $z^{k}$ is obtained after $x^{k}$ is calculated. If Condition 1 on the images of $A$ and $B$ is not satisfied, then it is possible that $x^k$ converges to some $x^*$ such that there is no $z^*$ satisfying $Ax^*+Bz^*=b$. In addition, Condition 2 provides a way to control dual iterates by primal iterates via the optimality condition of the $z$-subproblem. This relation requires unconstrained optimality condition of $z$-update, so the last \rev{block} variable $z$ cannot be constrained elsewhere. See also \cite{wang2015global} for some relevant discussions. As indicated from Table \ref{table:nonconvexADMM}, these two conditions (and their variants) are almost necessary for ADMM to converge in the absence of convexity. We also note that, even for convex problems, these two conditions are used to relax the strong convexity assumption in the objective \cite{lin2016iteration} or accelerate ADMM with $\mathcal{O}(1/k^2)$ iteration complexity \cite{tian2019alternating}.

It turns out that the two conditions and the two properties we mentioned in Section \ref{subsub: property} may conflict each other. By Property 1, the image of $A$ strictly constrains the image of $B$, so by Condition 1, we should update local variables after the global variable in each ADMM iteration to ensure feasibility. However, by Property 2, each local variable is subject to some local constraints, so Condition 2 cannot be satisfied; {technically speaking, we cannot utilize the unconstrained optimality condition of the last block to link primal and dual variables, which again makes it difficult to ensure primal feasibility of the solution}. When ADMM is directly applied to nonconvex problems, divergence is indeed observed \cite{sun2013fully,magnusson2015distributed,wang2015global}. {As a result, for many applications in the form of \eqref{formulation: generic} where the above two conditions are not available, the ADMM framework cannot guarantee convergence.}
	
After completing a draft of this paper, we were informed of a ADMM-based approach in \cite{jiang2019structured}, where the authors proposed to solve the \emph{relaxed} problem of \eqref{formulation: generic} 
	\begin{align}\label{eq: relaxed}
		\min_{x\in \mathcal{X}, \bar{x}\in \tilde{\mathcal{X}},z} \quad & f({x}) + \frac{\beta(\epsilon)}{2}\|z\|^2\quad \mathrm{s.t.}\quad
		 A{x}+B\bar{x}+z=0. 	 
	\end{align} 
Notice first that, as proved in \cite{jiang2019structured}, in order to achieve a desired feasibility  with $\|Ax+B\bar{x}\| = \mathcal{O}(\epsilon)$, the coefficient $\beta(\epsilon)$ and ADMM penalty need to be as large as $\mathcal{O}(1/\epsilon^2)$. Such large parameters may lead to slow convergence and large optimality gaps. Also notice that, applying ADMM to \eqref{eq: relaxed} may produce an approximate stationary solution to \eqref{formulation: generic}, even when the problem is infeasible to begin with.
{As we will show in Section \ref{section: Convergence}, our proposed two-level algorithm is able to achieve the same order of iteration complexity as the reformulation \eqref{eq: relaxed} and the one-level ADMM approach proposed in \cite{jiang2019structured}, and meanwhile the proposed two-level algorithm provides information on ill conditions and infeasibility; in Section \ref{section: Examples}, we empirically demonstrate with computation that the proposed algorithm robustly converges on large-scale constrained nonconvex programs with a faster speed and obtains solutions with higher qualities.}
	
\subsection{Other Distributed Algorithms}
Some other distributed algorithms not based on ADMM are also studied in the literature. Hong \cite{hong2016decomposing} introduced a proximal primal-dual algorithm for distributed optimization problems, where a proximal term is added to cancel out cross-product terms in the augmented Lagrangian function. Lan and Zhou \cite{lan2018random} proposed a randomized incremental gradient algorithm for a class of convex problems over a multi-agent network. Lan and Yang \cite{lan2018accelerated} proposed accelerated stochastic algorithms for nonconvex finite-sum and multi-block problems; interestingly, the analysis for the multi-block problem also requires the last block variable to be unconstrained with an invertible coefficient matrix and a Lipschitz differentiable objective, which further confirms the necessity of Conditions 1 and 2. We end this subsection with a recent work by Shi et al. \cite{shi2017penalty}. They studied the problem
\begin{align}\label{eq:pdd}
	\min_{\mathbf{x}, \mathbf{y}}~~f(\mathbf{x}, \mathbf{y}) + \sum_{j=1}^m \tilde{\phi}_j(\mathbf{y}_j)~~ \mathrm{s.t.}~~ h(\mathbf{x}, \mathbf{y}) = 0,~~ g_i(\mathbf{x}_i)\leq 0,~~\mathbf{x}_i\in \mathcal{X}_i~~ \forall i \in [n].
\end{align}
The variables $\mathbf{x}$ and $\mathbf{y}$ are divided into $n$ and $m$ subvectors, respectively. $f(\mathbf{x}, \mathbf{y})$, $h(\mathbf{x}, \mathbf{y})$, $g_i(\mathbf{x}_i)$ are continuously differentiable, $\tilde{\phi}_j(\mathbf{y}_j)$ is a composite function, and $\mathcal{X}_i$'s are convex. The authors proposed a doubly-looped penalty dual decomposition method (PDD). The overall algorithm used the ALM framework, where the coupling constraint $h(\mathbf{x}, \mathbf{y})=0$ is relaxed and each ALM subproblem is solved by a randomized block update scheme. We note that randomization is crucial in their convergence analysis, \rev{and a deterministic implementation of the inner-level algorithm for solving the ALM subproblem may not converge  when nonconvex functional constraints are present.}

\section{A Key Reformulation and A Two-level Algorithm}\label{sec:Reformulation}

We say $({x}^*, \bar{x}^*, y^*)\in \R^{n_1}\times\R^{n_2}\times\R^{m}$ is a 
\textit{stationary point} of problem \eqref{formulation: generic} if it satisfies the following condition
\begin{subequations}\label{condition: stationary}
	\begin{align}
	& 0\in \nabla f({x}^*)+ A^\top y^* + N_{\mathcal{X}}({x}^*), \label{condition: stationary1}\\
	& 0 \in B^\top y^* + N_{\bar{\mathcal{X}}}(\bar{x}^*),\label{condition: stationary2}\\
	& 0 = A{x}^*+B\bar{x}^*;\label{condition: stationary3}
	\end{align}
\end{subequations}
or equivalently, $0\in \partial L({x}^*, \bar{x}^*, y^*)$, where
\begin{equation}
L({x}, \bar{x}, y) := f(x) + \mathbb{I}_{\mathcal{X}}({x}) + \mathbb{I}_{\bar{\mathcal{X}}}(\bar{x})+ \langle y, A{x} + B\bar{x}\rangle.\label{eq: optcondtion}
\end{equation}
In equations \eqref{condition: stationary} and \eqref{eq: optcondtion}, the notation 
$N_{\mathcal{X}}(x)$ denotes the general normal cone of $\mathcal{X}$ at $x\in \mathcal{X}$ 
\cite[Def 6.3]{rockafellar2009variational}, and $\partial L(\cdot)$ denotes the general 
subdifferential of $L(\cdot)$ \cite[Def 8.3]{rockafellar2009variational}. Some properties and 
calculus rules of normal cones and the general subdifferential can be found in \cite[Chap 6, 8, 10]{rockafellar2009variational}. 

It can be shown that if $({x}^*, \bar{x}^*)$ is a local minimum of \eqref{formulation: generic} and satisfies some mild regularity condition, then condition \eqref{condition: stationary} is satisfied \cite[Thm 8.15]{rockafellar2009variational}. If $\mathcal{X}$ and $\bar{\mathcal{X}}$ are defined by finitely many continuously differentiable constraints, then condition \eqref{condition: stationary} is equivalent to the well-known KKT condition of problem \eqref{formulation: generic} under some constraint qualification. Therefore, condition \eqref{condition: stationary} can be viewed as a generalized first-order necessary optimality condition for nonsmooth constrained problems. Our goal is to find such a stationary point $({x}^*, \bar{x}^*, y^*)$ for problem \eqref{formulation: generic}.

\subsection{A Key Reformulation}
As analyzed in the previous section, since directly applying ADMM to a distributed formulation of the general constrained nonconvex problem \eqref{formulation: generic} cannot 
guarantee convergence without using the relaxation scheme in \cite{jiang2019structured}, we want to go beyond the standard ADMM framework. We propose two steps for achieving this. The first step is taken in this subsection to propose a new reformulation, and the 
second step is taken in the next subsection to propose a new two-level algorithm for the new reformulation.

We consider the following reformulation of \eqref{formulation: generic}
\begin{align}\label{formulation: slack}
	\min_{ x\in \mathcal{X}, \bar{x} \in \bar{\mathcal{X}},z} \quad f(x) \quad 	\mathrm{s.t.} \quad A{x} + B\bar{x} +z = 0, ~z=0.
\end{align}
The idea of adding a slack variable $z\in \R^m$ has two consequences. The first consequence is that the linear coupling constraint $Ax+B\bar{x}+z=0$ now has three blocks, and the last block is an identity matrix $I_m$, whose image is the whole space. {Given any ${x}$ and $\bar{x}$, we can always let $z = -Ax-B\bar{x}$ to make the constraint satisfied. The second consequence is that the artificial constraint $z=0$ can be treated separately from the coupling constraint.} Notice that a direct application of ADMM to problem \eqref{formulation: slack} still does not guarantee convergence since Conditions 1 and 2 are not satisfied yet.
So it is necessary to separate the linear constraints into two levels. If we ignore $z=0$ for the moment, existing techniques in ADMM analysis can be applied to the rest of the problem. Since we want to utilize the unconstrained optimality condition of the last block, we can relax $z=0$. This observation motivates us to choose ALM. To be more specific, consider the problem
\begin{align}\label{formulation: inner}
	\min_{ x\in \mathcal{X}, \bar{x} \in \bar{\mathcal{X}},z} \quad f(x)+\langle \lambda^k, z\rangle+ \frac{\beta^k}{2}\|z\|^2\quad \mathrm{s.t.} \quad A{x}+B\bar{x}+z=0,
\end{align}
which is obtained by dualizing constraint $z=0$ with  $\lambda^k\in \R^m$ and adding a quadratic penalty $\frac{\beta^k}{2}\|z\|^2$ with $\beta^k>0$. The augmented Lagrangian term $\langle \lambda^k, z\rangle+ \frac{\beta^k}{2}\|z\|^2$ can be viewed as an objective function in variable $z$, which is not only Lipschitz differentiable but also strongly convex. Problem \eqref{formulation: inner} can be solved by a three-block ADMM in a distributed fashion when a separable structure 
is available. Notice that the first-order optimality condition of problem \eqref{formulation: inner} at a stationary solution $({x}^{k}, \bar{x}^{k}, z^{k}, y^{k})$ is  
\begin{subequations}\label{condition: innerStationary}
	\begin{align}
	& 0\in \nabla f(x^{k}) +A^\top y^{k}+ N_{\mathcal{X}}({x}^{k}),\label{condition: innerStationary1}\\
	& 0\in B^\top y^{k} +N_{\bar{\mathcal{X}}}(\bar{x}^{k}), \label{condition: innerStationary2}\\
	& 0 =\lambda^k+\beta^kz^{k} +y^{k}, \label{condition: innerStationary3}\\
	& 0 = A{x}^{k}+B\bar{x}^{k} +z^{k}. \label{condition: innerStationary4}
	\end{align}
\end{subequations}
However, such a solution may not satisfy primal feasibility $Ax+B\bar{x}=0$, which is the only difference from the optimality condition \eqref{condition: stationary} (note that \eqref{condition: innerStationary3} is analogous to the dual feasibility in variable $z$ in the KKT condition). Fortunately, the ALM offers a scheme to drive the slack variable $z$ to zero by updating $\lambda$
and we can expect iterates to converge to a stationary point of the original problem \eqref{formulation: generic}. In summary, reformulation \eqref{formulation: slack} separates the complication of the original problem into two levels, where the inner level \eqref{formulation: inner} provides a formulation that simultaneously satisfies Conditions 1 and 2, and the outer level drives $z$ to zero. We propose a two-level algorithmic architecture in the next subsection to realize this.
\subsection{A Two-level Algorithm}
The proposed algorithm consists of two levels, both of which are based on the augmented Lagrangian framework. The inner-level algorithm is described in  Algorithm \ref{alg: inner}, which uses a three-block ADMM to solve problem \eqref{formulation: inner} and its iterates are indexed by $t$. The outer-level algorithm is described in Algorithm \ref{alg: outer} with iterates indexed by $k$.

Given $\lambda^k\in \R^m$ and $\beta^k>0$, the augmented Lagrangian function associated with the $k$-th inner-level problem \eqref{formulation: inner} is defined as 

\begin{align}
L_{\rho^k}({x},\bar{x}, z,y) :=& f(x) + \mathbb{I}_{\mathcal{X}}({x}) + \mathbb{I}_{\bar{\mathcal{X}}}(\bar{x}) + \langle\lambda^k, z\rangle+\frac{\beta^k}{2}\|z\|^2  \notag \\
 & + \langle y, A{x}+B\bar{x}+z\rangle + \frac{\rho^k}{2}\|A{x}+B\bar{x}+z\|^2,\label{eq: Lyap}
\end{align}
where $y \in \R^{m}$ is the dual variable for constraint $Ax+B\bar{x}+z=0$ and $\rho^k$ is a penalty parameter for ADMM.
{
In view of \eqref{condition: innerStationary}, the $k$-th inner-level ADMM aims to find an approximate stationary solution $(x^k, \bar{x}^k, z^k, y^k)$ of \eqref{formulation: inner}  in the sense that there exist $d_1^k$, $d_2^k$, and $d_3^k$ such that 
\begin{subequations}\label{eq: appStationary}
	\begin{align}
	& d^{k}_1\in \nabla f({x}^k) +A^\top y^{k}+N_{\mathcal{X}}({x}^k), \label{eq: appStationary1}\\
	& d^k_2 \in B^\top y^k+ N_{\bar{\mathcal{X}}}(\bar{x}^k), \label{eq: appStationary2}\\
	& 0 = \lambda^k + \beta^kz^k + y^k, \label{eq: appStationary3}\\
	& d^k_3 = A{x}^{k}+B\bar{x}^{k}+z^{k}, \label{eq: appStationary4}\\
	& \|d_i^k\|\leq \epsilon_i^k, \ \forall i\in [3],
	\end{align}
\end{subequations}
where $\epsilon_i^k$'s are positive tolerances. The optimality conditions of $x^{t}$ in Line \ref{a2: firstblock} and $\bar{x}^t$ in Line \ref{a2: secondblock} of Algorithm \ref{alg: inner} read:
\begin{align*}
	0 \in  & \nabla f(x^{t}) + A^\top y^{t-1} + \rho^k A^\top (Ax^{t} + B\bar{x}^{t-1} + z^{t-1}) + N_{\mathcal{X}}(x^t), \\
	0 \in  & B^\top y^{t-1} + \rho^k B^\top (Ax^{t} + B\bar{x}^{t} + z^{t-1}) + N_{\bar{\mathcal{X}}}(\bar{x}^t).
\end{align*}
With the dual update in Line \ref{a2: dual}, we can see that
\begin{align*}
	-\rho^k A^\top (B\bar{x}^{t-1}+z^{t-1}-B\bar{x}^{t}-z^{t}) \in &   \nabla f(x^{t}) + A^\top y^{t} +  N_{\mathcal{X}}(x^t), \\
	-\rho^k B^\top(z^{t-1}-z^{t})\in & B^\top y^t+ N_{\bar{\mathcal{X}}}(\bar{x}^t).
\end{align*}
As a result, Algorithm \ref{alg: inner} can be terminated if it finds $({x}^{t}, \bar{x}^{t}, z^{t})$ such that 
\begin{subequations}\label{stopping criteria}
	\begin{align} 
	\|\rho^k A^\top (B\bar{x}^{t-1}+z^{t-1}-B\bar{x}^{t}-z^{t})\| &\le \epsilon^{k}_1,\label{stopping criteria1}\\
	\|\rho^k B^\top(z^{t-1}-z^{t}) \|&\le \epsilon^{k}_2,\label{stopping criteria2}\\
	 \|A{x}^{t}+ B\bar{x}^{t}+z^{t}\|&\le \epsilon^{k}_3.\label{stopping criteria3}
	\end{align}
\end{subequations}
}Notice that $\rho^k$ does not appear in \eqref{stopping criteria3}, so we can use different tolerances for the above three measures. Since \eqref{eq: appStationary3} is always maintained by ADMM with \rev{$(y^k, z^k) = (y^t, z^t)$}, a solution satisfying \eqref{stopping criteria} is an approximate stationary solution to problem \eqref{formulation: inner} \rev{by assigning $(x^k, \bar{x}^k, z^k, y^k) := (x^t, \bar{x}^t, z^t, y^t)$.}
\begin{algorithm}[!htb]
	\caption{: The $k$-th inner-level ADMM}\label{alg: inner}
	\begin{algorithmic}[1]
		\STATE \rev{\textbf{Input} $(\lambda^k, \beta^k, \epsilon^{k}_1, \epsilon^{k}_2, \epsilon^{k}_3) \in \R^m \times \R^4_{++}$;}
		\STATE \rev{\textbf{initialize} $(x^0, \bar{x}^0, z^0, y^0)\in \mathcal{X} \times \bar{\mathcal{X}} \times \R^{m} \times \R^{m}$ with $\lambda^k+\beta^kz^0 +y^0=0$, $\rho^k = 2\beta^k$;}
		\FOR{\rev{$t = 1, 2, 3, \cdots$} }
 		\STATE /* First block update (parallelize over subvectors of ${x}$) */
		\STATE obtain a stationary $x^{t}$ such that $0\in \partial_x
		L_{\rho^k}({x}^{t},\bar{x}^{t-1},z^{t-1}, y^{t-1})$;\label{a2: firstblock}
 		\STATE /* Second block update (parallelize over components of $\bar{x}$) */
		\STATE $\bar{x}^{t} \gets \argmin_{\bar{x}}L_{\rho^k}({x}^{t}, \bar{x}, z^{t-1}, y^{t-1})$;\label{a2: secondblock}
 		\STATE /* Third block update (parallelize over subvectors of $z$) */
		\STATE $z^{t} \gets \argmin_{z}L_{\rho^k}({x}^{t}, \bar{x}^{t}, z, y^{t-1});$\label{a2: thirdblock}
 		\STATE /* Inner dual update (parallelize over subvectors of $y$) */
		\STATE $y^{t} \gets y^{t-1} +\rho^k(A{x}^{t}+ B\bar{x}^{t}+z^{t})$;\label{a2: dual}
		\IF{\rev{stopping criteria \eqref{stopping criteria} is satisfied}}
		\STATE  \rev{\textbf{return} $(x^t, \bar{x}^t, z^t, y^t)$;}
		\STATE \rev{\textbf{break}.}
		\ENDIF
		\ENDFOR
	\end{algorithmic}
\end{algorithm}

The first block update in Algorithm \ref{alg: inner} reads as
\begin{equation}\label{formulation: firstblock}
\min_{{x}\in \mathcal{X}} ~~f({x}) + \langle y^{t-1}, A{x}+B\bar{x}^{t-1}+z^{t-1}\rangle + \frac{\rho^k}{2}\|A{x}+B\bar{x}^{t-1}+z^{t-1}\|^2,
\end{equation}
so line \ref{a2: firstblock} of Algorithm \ref{alg: inner} searches for a stationary solution ${x}^{t}$ of the constrained problem \eqref{formulation: firstblock}. The second and third block updates in lines \ref{a2: secondblock} and \ref{a2: thirdblock} admit closed form solutions, so in view of the network flow problem \eqref{eq:networkExample}, the proposed reformulation \eqref{formulation: slack} does not introduce additional computational burden. All primal and dual updates in Algorithm \ref{alg: inner} can be implemented in parallel as $f$ and $\mathcal{X}$ admit separable structures. In each ADMM iteration, agents solve their own local problems independently and only need to communicate with their immediate neighbors. We resolve this by updating $\lambda$ and $\beta$, which is referred as outer-level iterations indexed by $k$ in Algorithm \ref{alg: outer}. 

\begin{algorithm}[!htb]
	\caption{: Outer-level ALM}\label{alg: outer}
	\begin{algorithmic}[1]
		\STATE \textbf{Initialize} \rev{$\lambda^1\in [\underline{\lambda}, \overline{\lambda}]$ where $\underline{\lambda}, \overline{\lambda}\in \R^{m}$ and $\overline{\lambda}-\underline{\lambda}\in \R^{m}_{++}$,  $\beta^1 = \beta^0 \gamma$ for some $\beta^0\geq \frac{1}{4}$ and $\gamma >1$, $\omega\in [0,1)$, $\{\epsilon^k_i\} \subset \R_{+}$ with $\epsilon^k_i \rightarrow 0$ for $i\in[3]$;}
		\FOR{\rev{$k=1,2,3,\cdots$}}
		\STATE \rev{obtain $({x}^{k},\bar{x}^{k},z^{k},y^{k})$ from Algorithm \ref{alg: inner} with input $(\lambda^k, \beta^k, \epsilon^{k}_1, \epsilon^{k}_2, \epsilon^{k}_3)$};
		\STATE $\lambda^{k+1} \gets \mathrm{Proj}_{[\underline{\lambda}, \overline{\lambda}]}\big(\lambda^{k} + \beta^k z^{k})$;
		\IF{$\|z^{k}\|\leq \omega\|z^{k-1}\|$}
		\STATE $\beta^{k+1} \gets \beta^k$,
		\ELSE
		\STATE $\beta^{k+1} \gets \gamma \beta^k$;
		\ENDIF
		\ENDFOR
	\end{algorithmic}
\end{algorithm}

In Algorithm \ref{alg: outer}, we choose some predetermined bounds $[\underline{\lambda}, \overline{\lambda}]$ and explicitly project the ``true'' dual variable $\lambda^k +\beta^k z^k$ onto this hyper-cube to obtain $\lambda^{k+1}$ used in the next outer iteration. Such safeguarding technique is essential to establish the global convergence of ALM \cite{andreani2007augmented,luo2008convergence}. We increase the outer-level penalty $\beta^k$ if there is no significant improvement in reducing $\|z^k\|$.

Before proceeding to the next section, we note that the key reformulation \eqref{formulation: slack} is inspired by the hope to reconcile the conflict between the two properties and the two condition so that ADMM can be applied. The introduction of additional variable $z$ is not necessary in the sense that any method that achieves distributed computation for the subproblem
\begin{align}\label{eq: alter_subproblem}
\min_{x\in \mathcal{X}, \bar{x} \in \bar{\mathcal{X}}} f(x)+ \langle \lambda^k,  Ax+B\bar{x} \rangle + \frac{\beta^k}{2}\|Ax+B\bar{x}\|^2
\end{align}
can be embedded inside the ALM framework. The aforementioned PDD method \cite{shi2017penalty} is such an approach. There are some other update schemes \cite{bolte2014proximal,xu2017globally} that can handle functional constraints in \eqref{eq: alter_subproblem}, {assuming that the (Euclidean) projection oracle onto the nonconvex set $\mathcal{X}$ is available.} It would be interesting to compare their performances with ADMM when used in the inner level, and we leave this to future work. {Meanwhile, as we will demonstrate in Section \ref{section: Examples}, the proposed two-level algorithm preserves the desirable properties of ADMM in practice, such as fast convergence in early stages and scalability to handle large-scale problems.}


\section{Global Convergence}\label{section: Convergence}
In this section, we prove global convergence and convergence rate of the proposed two-level algorithm. Starting from any initial point, iterates generated by the proposed algorithm have a limit point; every limit point is a stationary solution to the original problem under some mild condition. In particular, we make the following assumptions.
\begin{assumption}\label{assumption: feasible}
	Problem \eqref{formulation: slack} is feasible and the set of stationary points satisfying \eqref{condition: stationary} is nonempty.
\end{assumption}
\begin{assumption}\label{assumption:compact}
{
	The objective function $f: \R^n\rightarrow \R$ is continuously differentiable, $\mathcal{X}\subseteq\R^n$ is a compact set, and $\bar{\mathcal{X}}$ is convex and compact.}
\end{assumption}
\begin{assumption}\label{assumption: descentOverFirstBlock}
	Given $\lambda^k$, $\beta^k$, and $\rho^k$, the first block update can find a stationary solution $x^t$ such that 
	$0 \in \partial_x L_{\rho^k} (x^t, \bar{x}^{t-1}, z^{t-1},y^{t-1} )$ and
	$$L_{\rho^k} ({x}^{t}, \bar{x}^{t-1},z^{t-1}, y^{t-1}) \leq L_{\rho^k}( {x}^{t-1}, \bar{x}^{t-1},z^{t-1},y^{t-1})< +\infty$$ for all $t\in \Z_{++}$.
\end{assumption}

We give some comments below. Assumption \ref{assumption: feasible} ensures the feasibility of problem \eqref{formulation: slack}, which is standard. Though it is desirable to design an algorithm that can guarantee feasibility of the limit point, usually this is too much to ask: the powerful ALM may converge to an infeasible limit point even if the original problem is feasible. If this situation happens, or problem \eqref{formulation: slack} is infeasible in the first place, our algorithm will converge to a limit point that is stationary to some problem, as stated in Theorem \ref{thm: feasible}.
The compactness required in Assumption \ref{assumption:compact} ensures that the sequence generated by our algorithm stays bounded, and can be dropped if the existence of a limit point is directly assumed or derived from elsewhere. We do not make any explicit assumptions on matrices $A$ and $B$ in this section, and our analysis does not rely on any convenient structures that $A$ and $B$ may process, such as full row or column rank.

For Assumption \ref{assumption: descentOverFirstBlock}, we note that finding a stationary point usually can be achieved at the successful termination of some nonlinear solvers. {In addition, the state-of-the-art nonlinear solver IPOPT \cite{wachter2006implementation} will accept a trial point if either the objective or the constraint violation is decreased in each iteration. In step 1 of Algorithm \ref{alg: inner}, since ${x}^{t-1}$ is already a feasible solution, if we start from ${x}^{t-1}$, it is reasonable to expect a new stationary point ${x}^{t}$ is reached with an improved objective value.}
Assumption \ref{assumption: descentOverFirstBlock} is {slightly} weaker and more realistic than assuming that the nonconvex subproblem can be solved globally, {which is commonly adopted in the nonconvex ADMM literature}.

In Section \ref{section: inner}, we show that each inner-level ADMM converges to a solution that approximately satisfies the stationary condition \eqref{condition: innerStationary} of problem \eqref{formulation: inner}. This sequence of solutions that we obtain at termination of the inner ADMM is referred as outer-level iterates. Then in Section \ref{section: outer}, we firstly characterize limit points of outer-level iterates, whose existence is guaranteed. Then we show that a limit point is stationary to problem \eqref{formulation: generic} if some mild constraint qualification is satisfied.

\subsection{Convergence of Inner-level Iterations}\label{section: inner}
In this subsection, we show that, by applying the three-block ADMM to problem \eqref{formulation: inner}, we will get an approximate stationary point $({x}^{k}, \bar{x}^{k}, z^{k}, y^{k})$ satisfying \rev{the approximate stationary condition \eqref{eq: appStationary}.} The convergence of the inner-level ADMM in this subsection uses some techniques from the literature, e.g., \cite{wang2015global}. We present a self-contained proof \rev{in the appendix} and demonstrate that the descent oracle assumed in Assumption \ref{assumption: descentOverFirstBlock} relaxes the global optimality of subproblems without affecting the overall convergence. 
\rev{
\begin{proposition}\label{prop: innerAKKT}
	Suppose Assumptions \ref{assumption:compact}-\ref{assumption: descentOverFirstBlock} hold. The $k$-th inner-level ADMM of Algorithm \ref{alg: inner} terminates, i.e., the stopping criteria \eqref{stopping criteria} is satisfied, in at most
	\begin{align}
		T_k:= \left \lceil  \frac{8 \max\{\|A\|^2, \|B\|^2, 1\}\beta^k (\overline{L}_k - \underline{L})}{\min\{\epsilon^k_1,\epsilon^k_2, \epsilon^k_3 \}^2 } \right \rceil\label{eq:Tk}
	\end{align}
	iterations, where $\overline{L}_k := L_{\rho^k}(x^0, \bar{x}^0, z^0, y^0)$ and $\underline{L}\in \R$ is a finite constant independent of outer-level index $k$. 
\end{proposition}
\begin{proof}
	See Appendix \ref{sec: proof of prop1}. \qed 	
\end{proof}
In particular, the approximate stationary condition \eqref{eq: appStationary} is satisfied with the solution returned by ADMM.}

\subsection{Convergence of Outer-level Iterations}\label{section: outer}
In this subsection, we prove the convergence of outer-level iterations. In general, when the method of multipliers is used as a global method, there is no guarantee that the constraint being relaxed can be satisfied at the limit. Due to the special structure of our reformulation, we are able to give a characterization of limit points of outer-level iterates.

\begin{theorem}\label{thm: feasible}
	Suppose Assumptions \ref{assumption:compact}-\ref{assumption: descentOverFirstBlock} hold. Let $\{({x}^k, \bar{x}^k, z^k, y^k)\}_{k\in \Z_{++}}$ be the sequence of outer-level iterates of Algorithm \ref{alg: outer} satisfying condition \eqref{eq: appStationary}. Then the sequence of the primal solutions $\{({x}^k, \bar{x}^k, z^k)\}_{k\in \Z_{++}}$ are bounded, and  every limit point $({x}^*, \bar{x}^*, z^*)$ of this sequence satisfies one of the following:
	\begin{enumerate}
		\item $({x}^*, \bar{x}^*)$ is feasible for problem \eqref{formulation: generic}, i.e., $z^*=0$;
		\item $({x}^*, \bar{x}^*)$ is a stationary point of the problem 
		\begin{align}\label{formulation: feasibility}
		\min_{{x}\in \mathcal{X}, \bar{x}\in \bar{\mathcal{X}}} \quad \frac{1}{2}\|A{x}+B\bar{x}\|^2.
		\end{align}
	\end{enumerate}
\end{theorem}
\begin{proof}
	See Appendix \ref{sec: proof of thm1}. \qed	
\end{proof}

Theorem \ref{thm: feasible} gives a complete characterization of limit points of outer-level iterates. If the limit point is infeasible, i.e. $z^*\neq 0$, then \rev{$(x^*,\bar{x}^*)$} is a stationary point of the problem \eqref{formulation: feasibility}. This is also the case if problem \eqref{formulation: generic} is infeasible, i.e. the feasible region defined by $\mathcal{X}$ and $\bar{\mathcal{X}}$ does not intersect the affine plane $Ax+B\bar{x}=0$, since each inner-level problem \eqref{formulation: inner} is always feasible and the first case in Theorem \ref{thm: feasible} cannot happen. We also note that even if \rev{$(x^*,\bar{x}^*)$} falls into the second case of Theorem \ref{thm: feasible}, it is still possible that \rev{the associated} $z^*=0$, but then $(x^*, \bar{x}^*)$ will be some irregular feasible solution. In both cases, we believe $(x^*, \bar{x}^*)$ generated by the two-level algorithm has its own significance and may provide some useful information regarding the problem structure. Since stationarity and optimality are maintained in all subproblems, we should expect that any feasible limit point of the outer-level iterates is stationary for the original problem. As we will prove in the next theorem, this is indeed the case if some mild constraint qualification is satisfied.

\begin{theorem}\label{thm: convergence to stationary point}
	Suppose Assumptions \ref{assumption: feasible}-\ref{assumption: descentOverFirstBlock} hold. Let $({x}^*, \bar{x}^*, z^*)$ be a limit point of the outer-level iterates $\{({x}^{k}, \bar{x}^{k}, z^{k})\}_{k\in\Z_{++}}$ of Algorithm \ref{alg: outer}.
	 If $\{y^{k}\}_{k\in\Z_{++}}$ has a limit point $y^*$ along a subsequence converging to $({x}^*, \bar{x}^*, z^*)$, then $({x}^*, \bar{x}^*, y^*)$ is a stationary point of problem \eqref{formulation: generic} satisfying stationary condition \eqref{condition: stationary}.
\end{theorem}
\begin{proof}
	See Appendix \ref{sec: proof of thm2}. \qed
\end{proof}


In Theorem \ref{thm: convergence to stationary point}, we assume the dual variable $\{y^{k}\}$ has a limit point $y^*$. Since by \eqref{zoptCondition} we have $\lambda^k+\beta^kz^{k}+y^{k}=0$, the ``true" multiplier  $\tilde{\lambda}^{k+1}:=\lambda^k+\beta^kz^{k}$ also has a limit point. {We note that the existence of a limit point can be ensured by the existence of a bounded dual subsequence, which is known as the sequentially bounded constraint qualification (SBCQ) \cite{luo1996exact}. More specifically} in the context of smooth nonlinear problems, the constant positive linear dependence (CPLD) condition proposed by Qi and Wei \cite{qi2000constant} {also} guarantees that the sequence of dual variables has a bounded subsequence. Therefore, we think our assumption of $y^*$ is analogous to some constraint qualification in the KKT condition for smooth problems, and does not restrict the field where our algorithm is applicable.

We also give some comments regarding the predetermined bound $[\underline{\lambda},\overline{\lambda}]$ on outer-level dual variable $\lambda$. In principle, the bound should be chosen large enough at the beginning of the algorithm. Otherwise $\lambda^k$ will probably stay at $\underline{\lambda}$ or $\overline{\lambda}$ all the time; in this case, the outer-level ALM automatically converts to the penalty method, which usually requires $\beta^k$ to go to infinity, because, in general, exact penalization does not hold for a quadratic penalty function. In contrast, a proper choice of the dual variable can compensate asymptotic exactness even when the penalty function is not sharp at the origin. 
In terms of convergence analysis, one may notice that the choice of $\lambda$ is actually not that important: if we set $\lambda^k=0$ for all $k$, the analysis can still go through. This is because in the framework of ALM, the dual variable $\lambda$ is closely related to local optimal solutions. While we study global convergence, it is not clear which local solution the algorithm will converge to, so the role of $\lambda$ is not significant. {It seems difficult to establish the uniform boundedness of dual variables without the projection step, especially when there are nonconvex constraints}

In Section \ref{section: Local}, we will show our algorithm inherits some nice local convergence properties of ALM, {where $\lambda$ does play an important role}, and in Section \ref{section: Examples}, we will demonstrate that keeping $\lambda$ indeed enables the algorithm to converge faster than the penalty method.

\subsection{Iteration Complexity}\label{section: rate}
In this subsection, we provide an iteration complexity analysis of the proposed algorithm. \rev{In view of \eqref{condition: stationary}}, our goal is to give a complexity bound on the number of ADMM iterations for finding an $\epsilon$-stationary solution \rev{ $(x^K,\bar{x}^K, y^K)$} in the sense that there exist $d_1, d_2, d_3$ such that
\begin{subequations}\label{eq: opt_mod}
		\begin{align}
			&d_1\in \nabla f(x^K) + A^\top y^K+ N_\mathcal{X} (x^K), \label{eq: dual_orig_1}\\
			& d_2 \in B^\top y^K + N_{\bar{\mathcal{X}}} (\bar{x}^K), \label{eq: dual_orig_2}\\
			& d_3 = Ax^K+B\bar{x}^K,  \label{eq: primal_orig}\\
			& \max\{\|d_1\|,\|d_2\|,\|d_3\| \} \leq \epsilon.
		\end{align}
	\end{subequations}
In order to illustrate the main result in a concise and clear way, we slightly modify the outer-level Algorithm \ref{alg: outer} as follows.

\begin{algorithm}[!htb]
	\caption{: Modified Outer-level ALM}\label{alg: outer_mod}
	\begin{algorithmic}[1]
		\STATE \textbf{Initialize} \rev{$\lambda^1\in [\underline{\lambda}, \overline{\lambda}]$ where $\underline{\lambda}, \overline{\lambda}\in \R^{m}$ and $\overline{\lambda}-\underline{\lambda}\in \R^{m}_{++}$,  $\beta^1 = \beta^0\gamma$ for some $\beta^0\geq \frac{1}{4}$ and $\gamma >1$, $\epsilon > 0$;}
		\FOR{$k=1,2,3,\cdots$}
		\STATE \rev{obtain $({x}^{k},\bar{x}^{k},z^{k},y^{k})$ from Algorithm \ref{alg: inner} with input $(\lambda^k, \beta^k, \epsilon, \epsilon, \epsilon/2)$;}
		\STATE $\lambda^{k+1} \gets \mathrm{Proj}_{[\underline{\lambda}, \overline{\lambda}]}\big(\lambda^{k} + \beta^kz^{k})$, $\beta^{k+1}\gets\gamma\beta^{k}$;
		\ENDFOR
	\end{algorithmic}
\end{algorithm}

In Algorithm \ref{alg: outer_mod}, we choose some tolerance $\epsilon >0$ and apply the stopping criteria \eqref{stopping criteria} with \rev{$\epsilon^k_1=\epsilon^k_2 = 2\epsilon^k_3 = \epsilon$} for the $k$-th inner-level ADMM. For the ease of the analysis, we multiply the outer-level penalty $\beta^k$ by some $\gamma>1$ in each outer-iteration, instead of checking the improvement in primal feasibility.  Moreover, we add the following technical assumption.
\begin{assumption}\label{assumption: boundedLagrangian}
	There exists some $\overline{L}\in \R$ such that $L_{\rho^k} (x^0, \bar{x}^0,z^0, y^0) \leq \overline{L}$
	for all $k\in \Z_{++}$.
\end{assumption}
\begin{remark}
	This assumption can be satisfied if ADMM can make significant progress in reducing $\|z^k\|$ or equivalently $\|Ax^k + B\bar{x}^k\|$. Another naive implementation can be seen as follows: suppose a feasible point $(x, \bar{x})$ is known a priori, i.e., $(x, \bar{x})\in \mathcal{X}\times \bar{\mathcal{X}}$, and $Ax+B\bar{x}=0$,  then the initialization of the $k$-ADMM with $(x^0, \bar{x}^0, z^0, y^0) = (x, \bar{x},0, -\lambda^{k})$ guarantees that $L_{\rho^k} (x^0, \bar{x}^0, z^0,y^0)\leq \overline{L}$, where $\overline{L} = \max_{x\in \mathcal{X}} f(x)$.
\end{remark}

\begin{theorem}\label{thm: rate}
	Under Assumptions \ref{assumption: feasible}-\ref{assumption: boundedLagrangian}, Algorithm \ref{alg: outer_mod} finds an $\epsilon$-stationary solution \rev{$(x^K,\bar{x}^K, y^K)$} of \eqref{formulation: generic} in the sense of \eqref{eq: opt_mod} in no more than $\mathcal{O}\left(1/\epsilon^4\right)$ inner ADMM iterations. Furthermore, if $\hat{\lambda}^{k}:=\lambda^k + \beta^k z^k$ is bounded, then the \rev{iteration complexity} can be improved to $\mathcal{O}\left(1/\epsilon^3\right).$
\end{theorem}
\begin{proof}
	See Appendix \ref{sec: proof of thm3}.	\qed
\end{proof}

\rev{We acknowledge that $\{\hat{\lambda}^k\}_k$ may not be bounded for some applications. The second part of Theorem \ref{thm: rate} (as well as Theorem \ref{thm: multblock_rate} to be presented next) aims to reasonably justify the performance of the proposed algorithm under the boundedness condition. }

\subsection{Extension to Multi-block Problems}
In this section, we will discuss the extension of the two-level framework to the more general class of multi-block problems \eqref{eq:genericADMM}. In particular, we are interested in the case where Conditions 1 and 2 are not satisfied. As we mentioned earlier, Jiang et al. \cite{jiang2019structured} proposed to solve the following perturbed problem of \eqref{eq:genericADMM}:
	\begin{align}\label{eq: relaxed_mblock}
		\min_{x_1,\cdots, x_p, z}~&\sum_{i=1}^p f_p(x_i) + g(x_1,\cdots,x_p)+\lambda^\top z+\frac{\beta}{2}\|z\|^2\\
		~\mathrm{s.t.}~&\sum_{i=1}^pA_ix_i +z = b,~ x_i\in \mathcal{X}_i~\forall i\in [p]. \notag
	\end{align}
	for $\lambda =0$, where $f_i$'s are lower semi-continuous, and $f_p$ and $g$ are Lipschitz differentiable. Notice that we change $h$ and $B$ in \eqref{eq:genericADMM} to $f_p$ and $A_p$ for ease of presentation. The iteration complexity for this one-level workaround is $\mathcal{O}\left(1/\epsilon^4\right)$ when the dual variable is bounded, and  $\mathcal{O}\left(1/\epsilon^6\right)$ otherwise. In contrast, we can apply our two-level framework to the multi-block problem \eqref{eq:genericADMM} as well: with some initial guess $\lambda$ and moderate $\beta$, we solve \eqref{eq: relaxed_mblock} approximately using ADMM, and then we update $\lambda$ and $\beta$. \rev{We define dual residual similarly as in \eqref{stopping criteria1}-\eqref{stopping criteria2} for each block variable, and $\epsilon$-stationary solution as a pair of primal-dual points where the primal residual ($\|\sum_{i=1}^p A_ix_i-b\|$) and dual residuals (with respect to each primal block) are less than some $\epsilon>0$. An extension of the two-level framework is presented in Algorithm \ref{alg: mblock} below. }
	\begin{algorithm}[!htb]
	\caption{: Extension to Multi-block Problems}\label{alg: mblock}
	\begin{algorithmic}[1]
		\STATE \textbf{Initialize} \rev{$\lambda^1\in [\underline{\lambda}, \overline{\lambda}]$ where $\underline{\lambda}, \overline{\lambda}\in \R^{m}$ and $\overline{\lambda}-\underline{\lambda}\in \R^{m}_{++}$, $\beta^1 = \beta^0\gamma $ for some $\beta^0>0$ and $\gamma >1$, $\epsilon>0$; }
		\FOR{$k = 1,2,3,\cdots $}
		\STATE \rev{obtain an $(\epsilon/2)$-stationary solution $(x_1^k,\cdots, x_p^k,z^k, y^k)$ of \eqref{eq: relaxed_mblock} with $(\lambda, \beta) = (\lambda^k, \beta^k)$ by proximal ADMM-m or ADMM-g \cite{jiang2019structured};}
		\STATE $\lambda^{k+1} \gets \mathrm{Proj}_{[\underline{\lambda}, \overline{\lambda}]}\big(\lambda^{k} + \beta^kz^{k})$, $\beta^{k+1}\gets \gamma \beta^k$;
		\ENDFOR
	\end{algorithmic}
\end{algorithm}

\begin{theorem}\label{thm: multblock_rate}
	Under Assumption \ref{assumption: boundedLagrangian}, Algorithm \ref{alg: mblock} finds an $\epsilon$-stationary solution of \eqref{eq:genericADMM} in no more than $\mathcal{O}(1/\epsilon^6)$ ADMM iterations. Furthermore, if $\hat{\lambda}^{k}:=\lambda^k + \beta^k z^k$ is bounded, then the \rev{iteration complexity} can be improved to $\mathcal{O}\left(1/\epsilon^4\right).$
\end{theorem}
\begin{proof}
	See Appendix \ref{sec: proof of thm4}. \qed	
\end{proof}

\rev{Although the proposed algorithm invokes a series of ADMM with varying outer-level dual variables and penalties, Theorem \ref{thm: multblock_rate} suggests that its iteration complexity for finding a stationary solution is no worse than that of the single-looped ADMM variant proposed in \cite{jiang2019structured}. In Section \ref{section: Local}, local convergence results are presented as an alternative perspective to help us understand the behavior of the proposed algorithm.
}

\section{Local Convergence} \label{section: Local}
We show in this section that the proposed algorithm inherits some nice local convergence properties of the augmented Lagrangian method. 
{The analysis builds on the classic local convergence of ALM \cite{bertsekas2014constrained}, and our purpose is to provide some quantitative justification for the fast convergence of the two-level algorithm, which will be presented in Section \ref{section: Examples}.}

To begin with, we note that the inner-level problem \eqref{formulation: inner} solved by ADMM is closely related to the problem
\begin{align}\label{eq: inner_no_slack}
	\min_{x\in \mathcal{X}, \bar{x}\in \bar{\mathcal{X}}} f(x)- \langle \lambda^k, Ax+B\bar{x}\rangle +\frac{\beta^k}{2}\|Ax+B\bar{x}\|^2.
\end{align}
It is straightforward to verify that $(x^k, \bar{x}^k)$ is a stationary point of \eqref{eq: inner_no_slack} in the sense that 
\begin{subequations}\label{eq: inner_no_slack_stationary}
\begin{align}
	0 \in & \nabla f(x^k) + A^\top (-\lambda^k + \beta^k(Ax^k+B\bar{x}^k)) +N_{\mathcal{X}}(x^k),\\
	0 \in & B^\top (-\lambda^k + \beta^k(Ax^k+B\bar{x}^k)) +N_{\bar{\mathcal{X}}}(\bar{x}^k),
\end{align}
\end{subequations}
if and only if $(x^k,\bar{x}^k,z^k,y^k)$ is a stationary point of \eqref{formulation: inner} satisfying \eqref{condition: innerStationary} with $z^k = -Ax^k-B\bar{x}^k$ and $y^k = -\lambda^k + \beta^k(Ax^k+B\bar{x}^k)$. In addition, an approximate stationary solution of \eqref{formulation: inner} can be mapped to an approximate solution of \eqref{eq: inner_no_slack}.
\begin{lemma}
	Let $(x^k, \bar{x}^k, z^k, y^k)$ be a $(d_1^k,d_2^k,d_3^k)$-stationary point of \eqref{formulation: inner} in the sense of \eqref{eq: appStationary}. Then $(x^k, \bar{x}^k)$ is a $(\tilde{d}_i^k, \tilde{d}_2^k)$-stationary point of \eqref{eq: inner_no_slack}, i.e.,
\begin{subequations}\label{eq: inner_no_slack_stationary_approx}
\begin{align}
	\tilde{d}_1^k \in & \nabla f(x^k) + A^\top (-\lambda^k + \beta^k(Ax^k+B\bar{x}^k)) +N_{\mathcal{X}}(x^k),\\
	\tilde{d}_2^k \in & B^\top (-\lambda^k + \beta^k(Ax^k+B\bar{x}^k)) +N_{\bar{\mathcal{X}}}(\bar{x}^k),
\end{align}
\end{subequations}
where $ \tilde{d}_1^k = d_1^k + \beta^k A^\top d_3^k$, and $\tilde{d}_2^k = d_2^k + \beta^k B^\top d_3^k$.
\end{lemma}
\begin{proof}
	By \eqref{eq: appStationary3} and \eqref{eq: appStationary4}, we have $y^k = -\lambda^k +\beta^k (Ax^k+Bz^k-d_3^k)$; plugging this equality into 	\eqref{eq: appStationary1}-\eqref{eq: appStationary2} yields the result. \qed
\end{proof}

Thus we will mainly focus on problem \eqref{eq: inner_no_slack} and its approximate stationarity system \eqref{eq: inner_no_slack_stationary_approx} in this section. We add following assumptions on problem \eqref{formulation: generic}.
\begin{assumption}\label{assumption: second_order}
    {
	The set $\mathcal{X}=\{x\in \R^{n_1} : h(x)=0\}$ is compact with $h:\R^{n_1}\rightarrow \R^{p}$ being second-order continuously differentiable, the objective $f$ is second-order continuously differentiable over some open set containing $\mathcal{X}$, and $\bar{\mathcal{X}}$ is a convex set with nonempty interior in $R^{n_2}$. The matrix $B$ has full column rank. 
	}
\end{assumption}
\begin{remark}
	Any inequality constraint in $\mathcal{X}$ can be converted to the form $h(x)=0$ by adding the squares of additional slack variables. The second-order continuous differentiability of $f$ and $h$ are standard to establish local convergence of the augmented Lagrangian method. In addition, we explicitly require $B$ to have full column rank, which can be justified by the reformulation \eqref{eq:dupSchemeParallel}.
\end{remark}

\begin{definition}
	Let $x^* \in \mathcal{X}=\{x|h(x)=0\}$ and $\nabla h(x^*) = [\nabla h_1(x^*),\cdots, \nabla h_p(x^*)]\in \R^{n_1 \times p}$.
	\begin{enumerate}
		\item The tangent cone of $\mathcal{X}$ at $x^*$: 
		$$T_{\mathcal{X}}(x^*)=\left\{d\in \R^{n_1}~|~ \exists x^k\in \mathcal{X},  x^k\rightarrow x^*, \frac{x^k-x^*}{\|x^k-x^*\|}\rightarrow \frac{d}{\|d\|}\right\}.$$
		\item The cone of the first-order feasible variation of $\mathcal{X}$ at $x^*$: $$V_{\mathcal{X}}(x^*)=\{d \in \R^{n_1}: \nabla h(x^*)^\top d=0\}.$$
		\item We say that $x^*$ is quasiregular if $T_{\mathcal{X}}(x^*)=V_{\mathcal{X}}(x^*)$.
	\end{enumerate}

\end{definition}
\begin{assumption}\label{assumption: second_order_sufficient}
	Problem \eqref{formulation: generic} has a feasible solution $(x^*, \bar{x}^*)$, where $\bar{x}^*\in \mathrm{Int}~\bar{\mathcal{X}}$ and all equality constraints have linearly independent gradient vectors. In addition, $(x^*, \bar{x}^*)$, together with some dual multipliers $\lambda^*\in (\underline{\lambda}, \overline{\lambda})$ and $\mu^*\in \R^p$, satisfy
	\begin{subequations}
		\begin{align}
			& \nabla f(x^*) - A^\top \lambda^* + \nabla h(x^*)\mu^* = 0, ~~B^\top \lambda^* = 0,\\
			& u^\top \left(\nabla^2 f(x^*)+\sum_{i=1}^p \mu^*_i \nabla^2 h_i(x^*)\right)u >0, \notag \\
			& \quad\quad ~\forall (u,v)\neq 0 ~\mathrm{s.t.}~ Au+Bv =0, ~\nabla h(x^*)^\top u = 0.	\label{assumption: sosc}
		\end{align}
		\end{subequations}
	Moreover, there exists $R>0$ such that $x$ is quasiregular for all $x\in B_{R}(x^*) \cap \mathcal{X}$.
\end{assumption}
\begin{remark}
	Assumption \ref{assumption: second_order_sufficient} can be regarded as a second-order sufficient condition at a local minimizer $(x^*,\bar{x}^*)$ of problem \eqref{formulation: generic}, and $B$ having full column rank is necessary for \eqref{assumption: sosc} to hold. 
	The quasiregularity assumption can be satisfied by a wide range of constraint qualifications.
	\end{remark}

The quasiregularity condition bridges the normal cone stationarity condition to the well-known KKT condition.
\begin{proposition}
	If $x^k\in \mathcal{X}$ is quasiregular, $\bar{x}^k \in \mathrm{Int}~\bar{\mathcal{X}}$, and $(x^k, \bar{x}^k)$ satisfies condition \eqref{eq: inner_no_slack_stationary_approx}	with some $\tilde{d}_1^k$ and $\tilde{d}_2^k$, then there exists some $\mu^k \in \R^p$ such that $(x^k, \bar{x}^k)$ satisfies the approximate KKT condition of problem \eqref{eq: inner_no_slack}, i.e., $h(x^k)=0$,
	\begin{subequations}\label{eq: approx stationary}
	\begin{align}
		\tilde{d}_1^k &= \nabla f(x^k) + A^\top(-\lambda^k +\beta^k(Ax^k+B\bar{x}^k)) + \nabla h(x^k)\mu^k,\\
		\tilde{d}_2^k &= B^\top (-\lambda^k + \beta^k(Ax^k+B\bar{x}^k)).
	\end{align}
	\end{subequations}
\end{proposition}
\begin{proof}
     The claim uses the fact that the normal cone $N_{\mathcal{X}}(x)$ is the polar cone of the tangent cone $T_{\mathcal{X}}(x)$, and $N_{\bar{\mathcal{X}}}(\bar{x}) = \{0\}$ for $\bar{x}\in \mathrm{Int}~\bar{X}$. The existence of $\mu^k$ follows from the Farkas' Lemma \cite[Prop 4.3.12]{bertsekas1999nonlinear}. \qed
\end{proof}
\begin{proposition}\label{prop: existence of local sol}
	Suppose Assumption \ref{assumption: second_order} holds, and let $(x^*, \bar{x}^*, \mu^*, \lambda^*)$ be defined as in Assumption \ref{assumption: second_order_sufficient}. There exist positive $\underline{\beta}$ and $\delta$ such that for all $s=(\lambda, \beta, \tilde{d}_1,\tilde{d}_2)$ belonging to the set 
		$$S := \left \{s=(\lambda, \beta, \tilde{d}_1,\tilde{d}_2)~|~\left(\frac{\|\lambda-\lambda^*\|^2}{\beta^2} + \|\tilde{d}_1\|^2 + \|\tilde{d}_2\|^2\right)^{1/2} \leq \delta, \beta\geq \underline{\beta} \right \},$$
	there exist unique continuously differentiable mappings $x(s)$, $\bar{x}(s)$, $\mu(s)$, and $\tilde{\lambda}(s) = \lambda-\beta [Ax(s)+B\bar{x}(s)]$
		defined in the interior of $S$ satisfying 
		\begin{align}
			& \nabla f[x(s)] - A^\top \tilde{\lambda}(s) +\nabla h[x(s)] \mu(s) = \tilde{d}_1, ~B^\top \tilde{\lambda}(s) = \tilde{d}_2,~ h[x(s)]= 0;\label{eq: implicit_1}\\
			& \left(x(\lambda^*, \beta,0,0), \bar{x}(\lambda^*, \beta,0,0), \mu(\lambda^*, \beta,0,0), \tilde{\lambda}(\lambda^*, \beta,0,0)\right) = (x^*, \bar{x}^*, \mu^*, \lambda^*);\label{eq: implicit_2}\\
			& \bar{x}(s)\in\mathrm{Int}~\bar{\mathcal{X}}, \|x(s)-x^*\|\leq R.\label{eq: implicit_3}
		\end{align}
		Moreover, there exists $M>0$ such that for any $s\in S$, we have
		{
		\begin{align}
			& \max \{\|x(s)-x^*\|, \|\bar{x}(s)-\bar{x}^*\|, \|\tilde{\lambda}(s) - \lambda^*\| \} \notag \\
			 \leq & M(\|\lambda - \lambda^*\|^2/\beta^2 + \|\tilde{d}_1\|^2+\|\tilde{d}_2\|^2)^{1/2}.\label{eq: implicit_4} 
		\end{align}}
\end{proposition}
\begin{proof}
	See Appendix \ref{sec: proof of prop3}. \qed	
\end{proof}

\begin{proposition}\label{proposition: contraction}
Suppose Assumptions \ref{assumption: second_order} and \ref{assumption: second_order_sufficient} hold. Let $M$ and $S$ be defined as in Proposition \ref{prop: existence of local sol}. Suppose for some $(\beta^k, \lambda^k)$ with $\beta^k \geq M$, ADMM finds a $(d_1^k, d_2^k, d_3^k)$-stationary solution $(x^k, \bar{x}^k, z^k, y^k)$ satisfying \eqref{eq: appStationary} such that 
	\begin{enumerate}
		\item $s^k = (\lambda^k, \beta^k, \tilde{d}_1^k, \tilde{d}_2^k)\in S$, where $\tilde{d}_1^k = d_1^k + \beta^k A^\top d_3^k$, and $\tilde{d}_2^k = d_2^k + \beta^k B^\top d_3^k$;
		\item $(x^k, \bar{x}^k )= (x(s^k), \bar{x}(s^k))$;
		\item there exists a positive constant $\eta < \beta^k/M$ such that
		\begin{equation}\label{eq: ub on residual}
			   \left( \|A\|+\|B\|+\frac{1}{M} \right) (\|d_1^k\|+\|d_2^k\|+\|d_3^k\|)\leq \frac{\eta}{\beta^k} \|Ax^k+B\bar{x}^k\|.
		\end{equation}
	\end{enumerate}	
	Denote $\hat{\lambda}^k := \lambda^k + \beta^k z^k$. Then we have
	\begin{equation}\label{eq: contraction}
		\|\hat{\lambda}^k - \lambda^*\| \leq \left(\frac{M}{\beta^k} +\frac{M\eta (M+\beta^k)}{\beta^k(\beta^k-M\eta)} \right) \|\lambda^k - \lambda^*\|.
	\end{equation}
\end{proposition}
\begin{proof}
	See Appendix \ref{sec: proof of prop4}. \qed	
\end{proof}
	
\begin{theorem} \label{thm: local converge}
	Suppose Assumptions \ref{assumption: second_order} and \ref{assumption: second_order_sufficient} hold. Let 
	$\underline{\beta}$, $\delta$, $M$, and $S$ be defined as in Proposition \ref{prop: existence of local sol}. Suppose the three conditions in Proposition \ref{proposition: contraction} are satisfied for all iterates $k\in \Z_{+}$, and the initial penalty $\beta^0 > \frac{M}{\varrho}(1 +\eta + \varrho \eta)$ for some $\varrho \in (0,1)$. Then the following results hold:
	\begin{enumerate}
		\item the sequence $\{\lambda^k\}_{k\in\Z_{++}}$ stays inside the interior of $[\underline{\lambda},\overline{\lambda}]$, i.e., $\lambda^{k+1} = \hat{\lambda}^k = \lambda^k +\beta^k z^k$;
		\item the dual variable $\lambda^k$ converges to $\lambda^*$ with at least a linear rate i.e., 
			$$
			\lim_{k\rightarrow+\infty}\frac{\|\lambda^{k+1}-\lambda^*\|}{\|\lambda^{k}-\lambda^*\|}  \leq \varrho < 1, ~\text{and}~\lim_{k\rightarrow+\infty}\frac{\|\lambda^{k+1}-\lambda^*\|}{\|\lambda^{k}-\lambda^*\|}  =0 \text{~if~} \beta^k\rightarrow +\infty;$$
		\item 
			$\max \{\|x^k-x^*\|, \|\bar{x}^k-\bar{x}^*\|\} \leq \varrho \|\lambda^k-\lambda^*\| \leq \varrho^{k+1} \|\lambda^0-\lambda^*\|.$
	\end{enumerate}
\end{theorem}
\begin{proof}
	The coefficient in the right-hand side of \eqref{eq: contraction} is less than $\varrho$ if $\beta^k > \frac{M}{\varrho}(1 +\eta + \varrho \eta)$, and converges to 0 if $\beta^k\rightarrow +\infty$; thus the first two parts of the theorem are proved. Part 3 is due to \eqref{eq: implicit_4} and the same derivation as in Proposition \ref{proposition: contraction}. \qed 
\end{proof}

\rev{
Theorem \ref{thm: local converge} suggests that if we have a good initial point (inside the set $S$ defined in Proposition \ref{prop: existence of local sol}) and each inner ADMM locates the approximate stationary solution specified by the implicit function theorem (as in Proposition \ref{proposition: contraction}), then the two-level algorithm exhibits local linear or super-linear convergence in its outer level. The results are consistent with our empirical observations to be presented in Section \ref{section: Examples}, where usually only a few outer-level updates are required upon convergence.
}

\section{Examples}\label{section: Examples}
We present some applications of the two-level algorithm. All programs are coded using the Julia programming language 1.1.0 with JuMP package 0.18 \cite{DunningHuchetteLubin2017} and implemented on a 64-bit laptop with one 2.6 GHz Intel Core i7 processor, 6 cores, and 16GB RAM. 
All nonlinear constrained problems are solved by the interior point solver IPOPT (version 3.12.8) \cite{wachter2006implementation} with linear solver MA27.

\subsection{Nonlinear Network Flow Problem}\label{section: network}
We consider a specific class of network flow problems, which is covered by the motivating formulation \eqref{eq:networkExample}. Suppose a connected graph $G(\mathcal{V}, \mathcal{E})$ is given, where some nodes have demands of certain commodity and such demands need to be satisfied by some supply nodes. Each node $i$ keeps local variables $[p_i;x_i; \{x_{ij}\}_{j\in \delta(i)}; \{y_{ij}\}_{j\in \delta(i)}]$ $\in \R^{2|\delta(i)|+2}$. Variable $p_i$ is the production variable at node $i$, and $(x_i, x_{ij}, y_{ij})$ determine the flow from node $i$ to node $j$: $p_{ij} = g_{ij}(x_i, x_{ij}, y_{ij})$ where $g_{ij}: \R^3\rightarrow \R$. For example, in an electric power network or a natural gas network, variables $(x_i, x_{ij}, y_{ij})$ are usually related to electric voltages or gas pressures of local utilities. Moreover, for each $(i,j)\in \mathcal{E}$, nodal variables $(x_i,x_j, x_{ij}, y_{ij})$ are coupled together in a nonlinear fashion: $h_{ij}(x_i, x_j, x_{ij},y_{ij}) =0$ where $h_{ij}:\R^4\rightarrow \R$. As an analogy, this coupling represents some physical laws on nodal potentials. We consider the problem 
\begin{subequations} \label{formulation: socp}
\begin{align}
	\min \quad& \sum_{i\in \mathcal{V}}f_i(p_i) \\
	\mathrm{s.t.} \quad
	& p_i - d_i = \sum_{j\in\delta(i)} p_{ij}\quad \forall i \in \mathcal{V}, \label{eq: socp1}\\
	& p_{ij} = g_{ij}(x_i, x_{ij}, y_{ij})\quad \forall (i,j)\in \mathcal{E},\label{socp: balance}\\
	&h_{ij}(x_i, x_j, x_{ij},y_{ij})=0\quad \forall (i,j)\in \mathcal{E},\label{socp: couple}\\
	& x_i \in [\underline{x}_i, \overline{x}_i]\quad \forall i\in \mathcal{V}.
\end{align}
\end{subequations}
In \eqref{formulation: socp}, the generation cost of each node, denoted by $f_i(\cdot)$, is a function of its production level $p_i$. The goal is to minimize total generation cost over the network. Each node is associated with a demand $d_i$ and has to satisfy the injection balance constraint \eqref{eq: socp1}; nodal variable $x_i$ is bounded in $[\underline{x}_i, \overline{x}_i]$. Formulation \eqref{formulation: socp} covers a wide range of problems and can be categorized into the GNF problem studied in \cite{sojoudiconvexification}. 
Suppose the network is partitioned into a few subregions, and $(i,j)$ is an edge crossing two subregions with $i$ (resp. $j$) in region 1 (resp. 2). In order to facilitate parallel implementation, we replace constraint \eqref{socp: couple} by the following constraints with additional variables:
\begin{subequations}\label{eq: newCouple}
\begin{align}
	& h_{ij}(x^1_{i}, x^1_{j}, x_{ij}, y_{ij}) =0,~h_{ji}(x^2_{j}, x^2_{i}, x_{ji}, y_{ji}) =0,\\
	& x^1_i = \bar{x}_i,~x^2_i = \bar{x}_i, ~x^1_j = \bar{x}_j,~x^2_j = \bar{x}_j;
\end{align}
\end{subequations}
similarly, we replace $p_{ij}$ and $p_{ji}$ in \eqref{socp: balance} by
\begin{align}\label{eq: newFlow}
	&p_{ij} = g_{ij}(x_i^1, x_{ij}, y_{ij}),~p_{ji} = g_{ji}(x_j^2, x_{ji}, y_{ji}).
\end{align}
Notice that $(x_i^1, x_j^1, x_{ij}, y_{ij})$ are controlled by region 1 and $(x_i^2, x_j^2, x_{ji}, y_{ji})$ are controlled by region 2. After incorporating constraints \eqref{eq: newCouple}-\eqref{eq: newFlow} for all crossing edges $(i,j)$ into problem \eqref{formulation: socp}, the resulting problem is in the form of \eqref{formulation: generic} and ready for our two-level algorithm. We consider the case where coupling constraints are given by
$p_{ij} = \frac{a_i}{|\delta(i)|}x_i + b_{ij} x_{ij}+ c_{ij} y_{ij}$ and 
$h_{ij}(x_i,x_j, x_{ij}, y_{ij}) =x_{ij}^2+y_{ij}^2 - x_ix_j$. Constraint \eqref{socp: balance} is linear with parameters $(a_i, b_{ij}, c_{ij})$, while the nonconvex constraint \eqref{socp: couple} restricts $(x_i,x_j, x_{ij}, y_{ij})$ on the surface of a rotated second-order cone. 

We use the underlying network topology from \cite{zimmerman2011matpower} to generate our testing networks. Each network is partitioned into two, three, or four subregions. The graph information and centralized objectives from IPOPT are recorded in the first three columns of Table \ref{table:1}. The column ``LB" records the objective value by relaxing the constraint \eqref{socp: couple} to $h_{ij}(x_{i}, x_{j}, x_{ij}, y_{ij})\leq0$. It is clear that this relaxation makes problem \eqref{formulation: socp} convex and provides a lower bound to the global optimal value. Partition information are given in the last two columns.
\begin{table}[tbhp]
{\footnotesize
	\caption{Network information}\label{table:1}
\begin{center}
\begin{tabular}{ ccccccc }
	\hline
	$|\mathcal{V}|$ & 	$|\mathcal{E}|$&Central Obj. & LB& Idx& Partition Size & \# cross edges \\
	\hline
	    & 	  &		     &   	  &  14-2 &  5+9          & 3\\
	  14&   20&     53.67&	53.67 &  14-3 & 	 4+5+5        &	5\\
		& 	  &    	     &        &  14-4 &  2+4+4+4      &	7\\
	\hline
		&  	  &	 		 &        & 118-2 &	 47+71        & 4\\
	 118&  179&	   862.09&  862.03& 118-3 &  35+35+48     &	7\\
		& 	  &          &        & 118-4 &	 20+28+34+36  & 12\\
	\hline
		&  	  &	 		 &        & 300-2 &	 111+189      & 4\\
	 300&  409&	  4751.31& 4751.20& 300-3 &  80+87+133    &	7\\
		& 	  &          &        & 300-4 &	 58+64+88+90  & 11\\
	\hline
		&  	  &	 		 &        & 1354-2 &	 455+899          & 11\\
	1354&  1710&   740.09&  740.02& 1345-3 & 340+455+559      &	18\\
		& 	  &          &        & 1354-4 & 236+303+386+429  & 25\\
	\hline
\end{tabular}
\end{center}}
\end{table}
We compare our algorithm with PDD in \cite{shi2017penalty} as well as the proximal ADMM-g proposed in \cite{jiang2019structured} (which solves problem \eqref{eq: relaxed} instead). We set an absolute tolerance $\epsilon=1.0e-5$, {and initialize $(x_i, x_j, x_{ij},  y_{ij})$ with $(1,1,1,0)$ and $p_i$ with the initial value provided in \cite{zimmerman2011matpower}.}
For our two-level algorithm, we choose $\omega=0.75$, $\gamma=1.5$, and $\beta^1 = 1000$. Each component of $\lambda$ is restricted between $\pm 10^6$. The stopping criteria \eqref{stopping criteria} suggests that 
$\epsilon_1^k$ and $\epsilon^k_2$ should be of the order $\mathcal{O}(\rho^k\epsilon_3^k)$. Motivated by this observation,  we terminate the inner-level ADMM when $\|Ax^t+B\bar{x}^t +z^t\|\leq \max\{\epsilon, \sqrt{m}/(k\cdot \rho^k)\}$, where $m$ is the dimension of the vector, and $\rho^k$ is the inner ADMM penalty at outer iteration $k$. 
For PDD, as suggested in \cite[Section V.B]{shi2017penalty}, we terminate the inner-level of PDD when the relative gap of two consecutive augmented Lagrangian values is less than $\max\{\epsilon,100\epsilon \times (2/3)^k\}$; at the end of each inner-level rBSUM, the primal feasibility is checked and penalty is updated with the same $\omega$ and $\gamma$. Notice that the parameters used in the proposed algorithm and PDD are matched in our experiments. For proximal ADMM-g, we choose $\beta = 1/\epsilon^2$ and  $\rho = 3/\epsilon^2$; additional proximal terms $\frac{1}{2}\|x-x^t\|^2_H$ and $\frac{1}{2}\|\bar{x}-\bar{x}^t\|^2_H$ are added to the subproblem update, where $H = \frac{0.01}{\epsilon}I$. All three algorithms terminate if $\|Ax^k+B\bar{x}^k\| \leq \sqrt{m}\times \epsilon$. Test results are presented in Table \ref{table:2}.

\begin{table}[tbhp]
	{\footnotesize
		\caption{Comparison with PDD \cite{shi2017penalty}, proximal ADMM-g \cite{jiang2019structured}}\label{table:2}
	\begin{center}
	 \begin{tabular}{cccccccc}
		\hline
		Idx& 	Method&	Outer & Inner &	$\|A{x}+B\bar{x}\|$	& Obj	&Gap (\%)& Time (s) \\
		\hline
			&      ADMM-g&	  -&	   42&  3.35e-05&  93.06&   42.33&   8.30\\
		14-2&   	  PDD&   21&   94&	3.65e-05&  53.96& 	 0.53&   2.01\\
			&    Proposed&	 10&	   54&  3.77e-05&  53.98&	 0.58&   1.25\\

		\hline
			&     ADMM-g&	  -&	  247&  5.27e-05&  72.86&  26.34&   6.54\\
		14-3&   	 PDD&    22&  188&	3.88e-05&  53.98&   0.57&   1.82\\
		    & 	Proposed&	 20&  140&  1.11e-05&  53.99&   0.60&   1.40\\
		\hline
			&     ADMM-g&	  -&	  259&  5.90e-05&  81.67&  34.28&  7.58\\
		14-4&   	 PDD&    24&  896&	5.29e-05&  54.72& 	1.91&  9.41\\
		    & 	Proposed&	 19&  250&  7.69e-05&  54.42&   1.37&  2.43\\
	    \hline
	    	&     ADMM-g&	  -&	   40&  4.43e-05& 1283.48&  32.84&  6.20\\
	   118-2&   	 PDD&    24&   85&	3.34e-05&  870.20&   0.94& 	3.75\\
	        & 	Proposed&	 15&  100&  3.16e-05&  864.71&   0.31&  3.94\\
		\hline
			&     ADMM-g&	  -&	   67&  6.26e-05& 1200.01&  28.16&  1.91\\
	   118-3&   	 PDD&    25&  141&	5.80e-05&  867.44&   0.62& 	2.95\\
	        & 	Proposed&	 11&   86&  5.16e-05&  866.17&   0.48&  1.85\\	
		\hline
		    &     ADMM-g&	  -&   59&  8.11e-05& 1201.82&  28.27&   4.48\\
	   118-4&   	 PDD&    25&  178&	6.64e-05&  868.68&   0.77&   3.59\\
	        & 	Proposed&	 14&  137&  6.50e-05&  867.16&   0.59&   2.86\\
	    \hline
	    	&     ADMM-g&	  -&	  227&  4.80e-05&  5054.52&  6.00&   15.34 \\
	   300-2&   	 PDD&    28&   93&	3.87e-05&  4757.06&  0.12&   5.54\\
	        & 	Proposed&	 20&  304&  1.74e-05&  4751.71&  0.01&   18.04\\
		\hline
		    &     ADMM-g&	  -&	  400&  6.43e-05&  5049.16&  5.90&  20.46\\
	   300-3&   	 PDD&    28&  127&	6.27e-05&  4757.63&  0.13&  6.27\\
	        & 	Proposed&	 25&  517&  4.80e-05&  4752.52&  0.03&  23.83\\	
		\hline
		    &     ADMM-g&	  -&	 1000&  1.67e-04&  5041.50&  5.76&  46.41 \\
	   300-4&   	 PDD&    28&  243&	7.37e-05&  4765.06&  0.29&  10.28\\
	        & 	Proposed&	 20&  512&  7.26e-05&  4752.56&  0.03&  19.53\\
	   \hline
	   		&     ADMM-g&	  -&	  901&  7.86e-05& 767.44&   3.57&  672.51\\
	  1354-2&   	 PDD&    25&  299&	6.56e-05& 745.50& 	0.73&  212.91\\
	        & 	Proposed&	 19&  126&  6.39e-05& 743.32&   0.44&  84.61\\
	   \hline
	   		&     ADMM-g&	  -&	 1000&  1.66e-04&  771.52&  4.08&  342.77\\
	  1354-3&   	 PDD&    26&  422&  8.86e-05&  747.90& 	1.05&  174.78\\
	        & 	Proposed&	 18&  137&  7.04e-05&  744.91&  0.66&  50.77\\
	   \hline
	   		&     ADMM-g&	  -&	 1000&  4.90e-04&  769.55&  3.84&  265.59\\
	  1354-4&   	 PDD&    27&  838&	1.10e-04&  749.61& 	1.28&  523.78\\
	        & 	Proposed&	 18&  170&  8.15e-05&  744.98&  0.67&  115.71\\
	   \hline
	\end{tabular}
	\end{center}}
\end{table}
\rev{The number of outer-level updates (ALM multiplier updates for PDD and the two-level algorithm) and the total number of inner-level updates (rBSUM iterations for PDD and ADMM iterations for the two-level algorithm) are reported in columns ``Outer" and ``Inner", respectively. } We see that both the proposed algorithm and PDD converge in all test cases, and both of them take around 10-30 outer-level iterations to drive the constraint violation \rev{``$\|Ax+B\bar{x}\|$"} close to zero. 
PDD converges fast for three cases of network 300; however, for most cases it requires more total inner and outer iterations for convergence than the proposed algorithm. Such performance is consistent with the analysis in \cite{shi2017penalty}, where the inner-level rBSUM algorithm needs to run long enough to guarantee each block variable achieves stationarity. The objective values and duality gaps of solutions generated by the three algorithms are recorded in ``Obj" and ``Gap (\%)". We can see both the proposed algorithm and PDD are able to achieve near global optimality, while the proposed algorithm finds solutions with even higher quality than PDD at termination. The algorithm running time (model building time excluded) is recorded in the last column ``Time (s)". We would like to emphasize that, under similar algorithmic settings, the proposed two-level algorithm in general converges faster and shows better scalability than the other two algorithms. 

Even with sufficiently large penalty on the slack variable $z$, the proximal ADMM-g does not achieve the desired primal feasibility for cases 300-4, 1354-3, and 1354-4 in 1000 iterations; for other cases, it usually \rev{takes} more time than the proposed algorithm. We point out that ADMM-g usually finds sub-optimal solutions, and the duality gap can be as large as 42\%. We believe this happens because problem \eqref{eq: relaxed} requires the introduction of large $\beta(\epsilon)$ and $\rho(\epsilon)$, which affect the structure of the original problem \eqref{formulation: generic} and result in solutions with poor quality. Moreover, such large parameters also cause numerical issues for the IPOPT solver and slow down the overall convergence, and this is the reason why ADMM-g takes a long time even when the number of iterations is relatively small for the first four test cases. We also tried a smaller penalty $\mathcal{O}(1/\epsilon)$, in which case the ADMM-g cannot achieve the desired feasibility level.

\subsection{Minimization over Compact Manifold}\label{subsection: manifold}
We consider the following problem
\begin{subequations}\label{formulation: sphere}
	\begin{align}
		\min \quad & \sum_{i=1}^{n_p-1}\sum_{j=i+1}^{n_p} \big((x_i-x_j)^2 +(y_i-y_j)^2+(z_i-z_j)^2\big)^{-\frac{1}{2}}\label{sphere: obj}\\
		\mathrm{s.t.}\quad 
		& x_i^2+y_i^2+z_i^2=1,\quad \forall i \in [n_p].\label{sphere: constr}
	\end{align}
\end{subequations}
Problem \eqref{formulation: sphere} is obtained from the benchmark set COPS 3.0 \cite{dolan2004benchmarking} of nonlinear optimization problems. The same problem is used in \cite{wen2013feasible} to test algorithms that preserve spherical constraints through curvilinear search. We compare solutions and computation time of our distributed algorithm with those obtained from the centralized IPOPT solver. Each test problem is firstly solved in a centralized way; objective value and total running time are recorded in the second and third column of Table \ref{table:3}. Using additional variables to break couplings in the objective \eqref{sphere: obj}, we divide each test problem into three subproblems. Subproblems have the same number of variables, constraints, and objective terms (as in \eqref{sphere: obj}). For our two-level algorithm, we choose $\gamma = 2$, $\omega = 0.5$; initial value of penalty $\beta^1$ is set to 100 for $n_p \in \{60,90\}$, 200 for $n_p \in\{120,180\}$, and 500 for $n_p\in\{240, 300\}$. {The initial point is set to $(x_i,y_i,z_i)=(0.2, 0.3, 0.1)$ for all $i\in [n_p]$ for IPOPT.}
We set bounds on each component of $\lambda$ to be $\pm 10^6$. The inner-level ADMM terminates when $\|Ax^t+B\bar{x}^t+z^t\|\leq \sqrt{3n_p}/(2500k)$, where $k$ is the current outer-level index; the outer level terminates when $\|Ax^k+B\bar{x}^k\|\leq \sqrt{3n_p}\times1.0e-6$.
\begin{table}[h!]
	{\footnotesize
		\caption{Comparison of centralized and distributed solutions}\label{table:3}
	\begin{center}
\begin{tabular}{ccc|c>{\rowmac}c>{\rowmac}c>{\rowmac}c<{\clearrow}cc}
		\hline
		 &\multicolumn{2}{c|}{Centralized Ipopt} & \multicolumn{6}{c}{Proposed two-level algorithm and penalty method}\\
		\hline
		$n_p$ &   Obj. & Time (s)&   Method&            Outer&  Inner& $\|A{x}+B\bar{x}\|$  & Gap (\%) &Time (s)\\
		\hline
			60& 1543.83&    9.55&   Proposed& 11&    62& 1.17e-05&   0.79& {4.17} \\
			  &        &        &   Penalty&  18&   102& 1.32e-05&   0.54&   7.82 \\
		\hline
			90&	3579.18&   17.34&  Proposed& 12&    98& 1.01e-05&  0.14& {20.42}\\
			  &       &         &  Penalty&  18&   136& 9.62e-06&  0.13& 26.97\\
		\hline
		   120&	6474.77&  56.64&  Proposed&  12&    79& 8.77e-06&    0.30& {45.28}\\
		      &       &         &  Penalty&   17&   113& 1.75e-05&   0.21& 60.42\\
		\hline
		   180& 14867.41& 212.95& Proposed& 12&   82& 1.71e-05&    0.10& {173.81}\\
		      &        &         & Penalty& 18&  121& 1.69e-05&  0.09& 233.75\\
		\hline
		   240&26747.84&  710.68&  Proposed& 12&   79& 1.25e-05&  0.44& {417.62}\\
		   	  &        &        &   Penalty& 17&  111& 2.02e-05&  0.28& 534.59\\	
		\hline
		   300&42131.88& 1568.64&  Proposed& 12&   80& 1.51e-05&  0.17& {852.19}\\
		      &        &        &  Penalty& 18&  115& 2.94e-05&   0.12& 1094.91\\
		\hline
	\end{tabular}
	\end{center}}
\end{table}

The quality of the centralized solution is slightly better than distributed solutions, while our proposed algorithm is able to reduce the running time significantly except for one case ($n_p=90$) while ensuring feasibility. In addition, as indicated in Table \ref{table:3}, numbers of iterations for both inner and outer levels stay stable across all test cases, which suggests that the proposed algorithm scales well with the size of the problem. \rev{In view of the discussion in Section \ref{section: outer}, we compare with the penalty method, where $\lambda^k=0$ for all $k$, to demonstrate the effect of the outer-level dual variable.} Without updating $\lambda$, the penalty method requires more inner/outer updates and substantially longer time.

\subsection{A Multi-block Problem: Robust Tensor PCA}
In this section, we use the robust tensor PCA problem considered in \cite{jiang2019structured} to illustrate that the two-level framework can be generalized to multi-block problem \eqref{eq:genericADMM}, and {when Conditions 1 and 2 are satisfied, the resulting two-level algorithm can potentially accelerate one-level ADMM}. In particular, given an estimate $R$ of the CP-rank, the problem of interest is casted as 
	\begin{align}\label{eq: pca}
		\min_{A,B,C,\mathcal{Z}, \mathcal{E}, \mathcal{B}}	\|\mathcal{Z}-\llbracket A, B, C\rrbracket \|^2 +\alpha \|\mathcal{E}\|_1 + \alpha_N \|\mathcal{B}\|_F^2 \quad \mathrm{s.t.}~~ \mathcal{E}+\mathcal{Z}+\mathcal{B} = \mathcal{T},
	\end{align}
	where $A \in \mathbb{R}^{I_{1} \times R}, B \in \mathbb{R}^{I_{2} \times R}, C \in \mathbb{R}^{I_{3} \times R}$, and $\llbracket A, B, C \rrbracket$ denotes the sum of column-wise outer product of $A$, $B$, and $C$. We denote the mode-$i$ unfolding of tensor $\mathcal{Z}$ by  $Z_{(i)}$, the Khatri-Rao product of matrices by $\odot$, the Hadamard product by $\circ$, and the soft shrinkage operator by $\mathbf{S}$. We implement the two-level framework as in Algorithm \ref{alg: PCA}.
	\begin{algorithm}[!htb]
	\caption{: Two-level Algorithm for Robust Tensor PCA}\label{alg: PCA}
	\begin{algorithmic}[1]
		\STATE \textbf{Initialize} primal variables $A^0, B^0,C^0,\mathcal{E}^0, \mathcal{Z}^0, \mathcal{B}^0$, $\mathcal{S}^0$, dual variables $Y^0$, $\Lambda^0$, penalty parameters $\beta$, $\rho=c\beta$, stepsize $\tau = \frac{1}{\rho}$, constants $\delta_i>0$ for $i\in[6]$, $\gamma >1$;
		\FOR{\rev{$k=0,1,2, \cdots$}}
		\STATE \small{$A^{k+1} =[(Z)^k_{(1)}(C^k \odot B^k) + \frac{\delta_1}{2}A^k][((C^k)^\top C^k)\circ((B^k)^\top B^k) +\frac{\delta_1}{2} I_R]^{-1}$;} \label{pca_1}
		\STATE  \small{$B^{k+1} =[(Z)^k_{(2)}(C^k \odot A^k) + \frac{\delta_2}{2}B^k][((C^k)^\top C^k)\circ((A^k)^\top A^k) +\frac{\delta_2}{2} I_R]^{-1}$;}
		\STATE \small{$C^{k+1} =[(Z)^k_{(3)}(B^k \odot A^k) + \frac{\delta_2}{2}C^k][((B^k)^\top B^k)\circ((C^k)^\top C^k) +\frac{\delta_3}{2} I_R]^{-1}$;}
		\STATE \small{$E_{(1)}^{k+1}=\mathbf{S}\left(\frac{\rho}{\rho+\delta_{4}}\left(T_{(1)}+\frac{1}{\rho} Y_{(1)}^{k}-B_{(1)}^{k}-Z_{(1)}^{k} - S^k_{(1)}\right)+\frac{\delta_{4}}{\rho+\delta_{4}} E_{(1)}^{k}, \frac{\alpha}{\rho+\delta_{4}}\right)$;}\label{pca_4}
		\STATE \small{$Z_{(1)}^{k+1}=\frac{1}{2+2 \delta_{5}+\rho}\left(2 A^{k+1}\left(C^{k+1} \odot B^{k+1}\right)^{\top}+2 \delta_{5}Z_{(1)}^{k}+\Lambda_{(1)}^{k}-\rho\left(E_{(1)}^{k+1}+B_{(1)}^{k} +S_{(1)}^k-T_{(1)}\right)\right)$};\label{pca_5}
		\STATE $B^{k+1}_{(1)} = \frac{1}{2\alpha_N+\delta_6+\rho}(Y^k_{(1)}+\delta_6B^{k}_{(1)}-\rho\left(E^{k+1}_{(1)}+Z^{k+1}_{(1)}+S^k_{(1)}-T_{(1)} )\right)$; \label{pca_6}
		\STATE \small{$S^{k+1}_{(1)} = S^k_{(1)}-\tau (-Y^k_{(1)}+\Lambda_{(1)} + \rho(Z_{(1)}^{k+1}+E_{(1)}^{k+1} +B_{(1)}^{k+1}+S_{(1)}^{k}-T_{(1)}))$;}\label{pca_7}
		\STATE \small{$Y^{k+1}_{(1)}=Y^{k}_{(1)}-\rho\left(Z_{(1)}^{k+1}+E_{(1)}^{k+1} +B_{(1)}^{k+1}+S_{(1)}^{k+1}-T_{(1)}\right)$;} \label{pca_y}
		\IF{$\|\mathcal{Z}^{k+1}+\mathcal{E}^{k+1}+\mathcal{B}^{k+1} +\mathcal{S}^{k+1}-\mathcal{T}\|_F$ is smaller than some threshold}
		\STATE $\Lambda \gets \Lambda + \beta \mathcal{S}^{k+1}$, $\beta \gets \gamma\beta$, $\rho\gets c \beta$, $\tau\gets 1/\rho$; \label{pca_lmd}
		\ENDIF
		\ENDFOR
	\end{algorithmic}
	\end{algorithm}
	We firstly perform ADMM-g in steps \ref{pca_1}-\ref{pca_y}. We note that there are some modifications to the ADMM-g described in \cite{jiang2019structured}: since our two-level framework requires the introduction of an additional slack variable $\mathcal{S}$, steps \ref{pca_4}-\ref{pca_6} have an additional term $S^{k}_{(1)}$, and $S_{(1)}^{k+1}$ is then updated in step \ref{pca_7} {via a gradient step as in ADMM-g}; moreover, during the update of $\mathcal{B}$, we also add a proximal term with coefficients $\delta_6/2$. When the residual $\|\mathcal{Z}^{k+1}+\mathcal{E}^{k+1}+\mathcal{B}^{k+1} +\mathcal{S}^{k+1}-\mathcal{T}\|_F$ is small enough, which can serve as an indicator of the convergence of ADMM-g, {we multiply the penalty $\beta$ by some $\gamma$ {as long as $\beta<1.0e+6$}, and update the outer-level dual variable $\Lambda$ as in step \ref{pca_lmd}, where the projection step is omitted.}
	
	{
	We experiment on tensors with dimensions $I_1 = 30$, $I_2 = 50$, and $I_3 = 70$, which match the largest instances tested by \cite{jiang2019structured}; the initial estimation $R$ is given by $R_{CP}+ \lceil 0.2*R_{CP} \rceil$. In our implementation, we set $\gamma=1.5$, $c=3$, and the initial $\beta$ is set to 2; the inner-level ADMM-g terminates if the residual $\|\mathcal{Z}^{k+1}+\mathcal{E}^{k+1}+\mathcal{B}^{k+1} +\mathcal{S}^{k+1}-\mathcal{T}\|_F$ is less than $\max\{1e-5, 1e-3/K_{\text{out}}\}$, where $K_{\text{out}}$ is the current outer-level iteration count. All {other parameters, generation of problem data, and initialization follow} the description in \cite[Section 5]{jiang2019structured}. For each value of the CP rank, we generate 10 cases and let ADMM-g and the proposed two-level Algorithm perform 2000 (inner) iterations. We calculate the geometric mean $r^{\text{Geo}}_k$ of the primal residuals $r_k=\|\mathcal{Z}^{k}+\mathcal{E}^{k}+\mathcal{B}^{k}-\mathcal{T}\|_F$ over 10 cases, and plot $\lg r^{\text{Geo}}_k$ as a function of iteration count $k$ in Figure \ref{fig:res_pca}. We also calculate the geometric mean $e^{\text{Geo}}_k$ of relative errors $\|\mathcal{Z}^k-\mathcal{Z}_{\mathrm{true}}\|_F/\|\mathcal{Z}_{\mathrm{true}}\|_F$ over 10 cases, where $\mathcal{Z}_{\mathrm{true}}$ is the generated true low-rank tensor, and plot $e^{\text{Geo}}_k$ in Figure \ref{fig:res_pca_err}. For our two-level algorithm, the primal residual decreases relatively slow during the first few inner ADMM-g; however, as we update the outer-level dual variable $\Lambda$ and penalty $\beta$, $r^{\text{Geo}}_k$ drops significantly faster than that of ADMM-g, and achieves feasibility with high precision in around 500 inner iterations. The relative error $r^{\text{Geo}}_k$ of the two-level algorithm converges slightly slower than ADMM-g, while it is able to catch up and obtain the same level of optimality.} The result suggests that our proposed two-level algorithm not only ensures convergence for a wider range of applications where ADMM may fail, but also accelerates ADMM on problems where convergence is already guaranteed. 
	 

	\begin{figure}
		\centering
		\caption{Comparison of Infeasibility $\lg r^{\text{Geo}}_k$}\label{fig:res_pca}
		\includegraphics[scale=0.6]{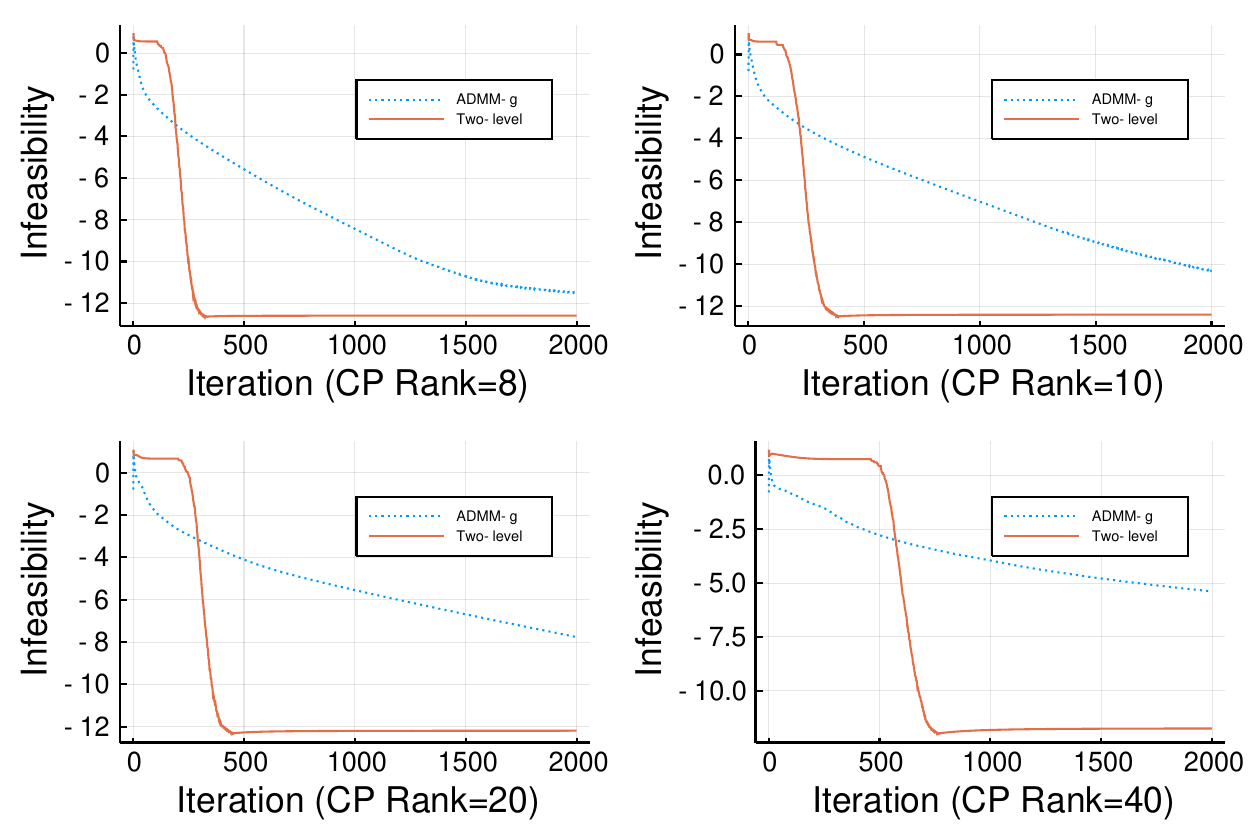}\\
	\end{figure}
	\begin{figure}
		\centering
		\caption{Comparison of Relative Errors $e^{\text{Geo}}_k$}\label{fig:res_pca_err}
		\includegraphics[scale=0.6]{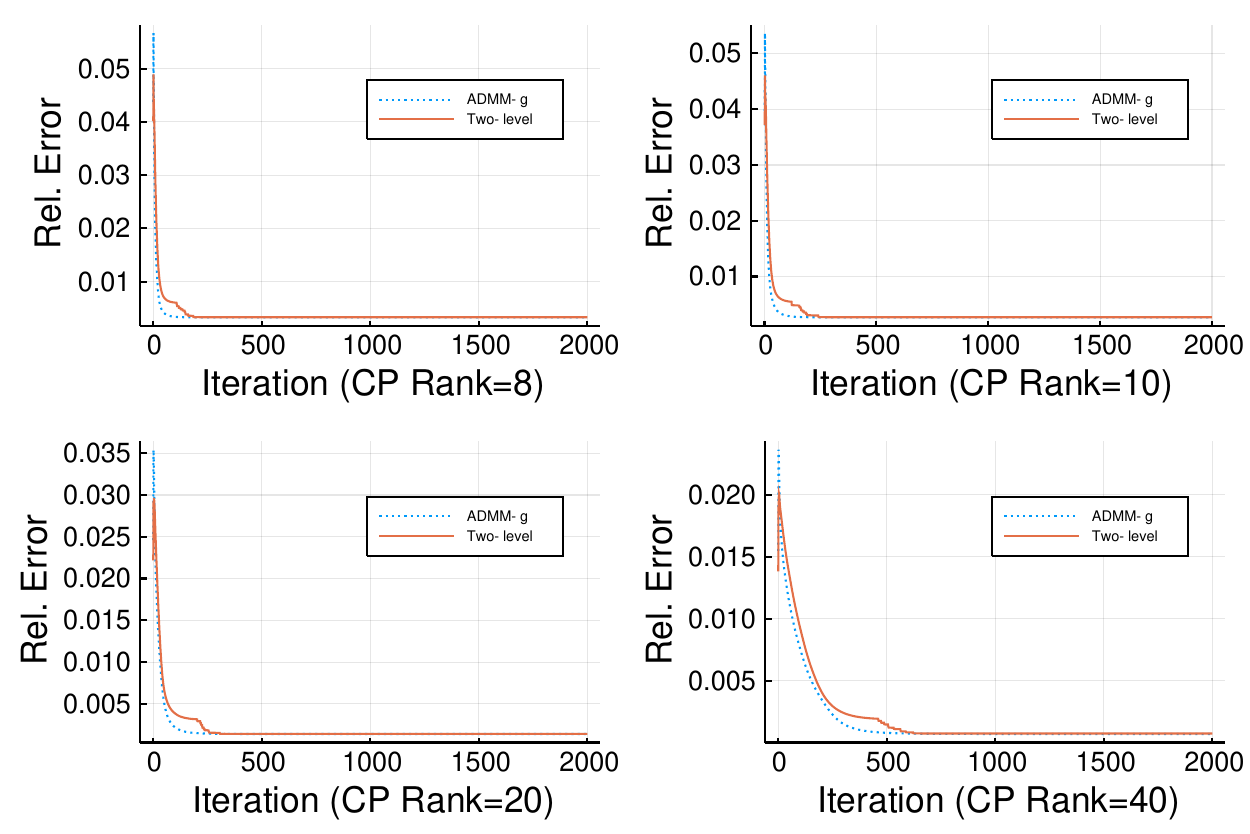}\\
	\end{figure}
	
%

\section{Conclusion}\label{section: Conclusion}
This paper proposes a two-level distributed algorithm to solve the nonconvex constrained optimization problem \eqref{formulation: generic}. We identify some limitation of the standard ADMM algorithm, which in general cannot guarantee convergence when parallelization of constrained subproblems is considered. In order to overcome such difficulties, we propose a novel while concise distributed reformulation, which enables us to separate the underlying complication into two levels. The inner level utilizes multi-block ADMM to facilitate parallel implementation while the outer level uses the classic ALM to guarantee convergence to feasible solutions. 
{
Global convergence, local convergence, and iteration complexity of the proposed two-level algorithm are established, and we certify the possibility to extend the underlying algorithmic framework to solve more complicated nonconvex multi-block problems \eqref{eq:genericADMM}.} In comparison to the other existing algorithms that are capable of solving the same class of nonconvex constrained programs, the proposed algorithm exhibits its advantages in terms of speed, scalability, and robustness. 
{
Thus for general nonconvex constrained multi-block problems, the two-level algorithm can serve an alternative to the workaround proposed in \cite{jiang2019structured} when Condition 1 or 2 fails, and potentially accelerate ADMM on problems where slow convergence is frequently encountered.
}
%
%
%


%
%

\bibliographystyle{spmpsci}    
\bibliography{references}
\appendix
\section{Additional Proofs in Section \ref{section: Convergence}}
\subsection{Proof of Proposition \ref{prop: innerAKKT}} \label{sec: proof of prop1}
\rev{We omit the index $k$ in $(\rho^k,\beta^k,\lambda^k,T_k)$ occasionally. We first prove two lemmas. } 
\begin{lemma}\label{lemma:closedFormSol}
	For all $t \in \Z_{++}$, we have 
	\begin{equation}\label{x_2optCondition}
	\langle B^\top y^{t-1}+ \rho B^\top(A{x}^{t}+B\bar{x}^{t}+z^{t-1}), \hat{x}-\bar{x}^t \rangle\geq 0\quad \forall \hat{x}\in \bar{\mathcal{X}},
	\end{equation} 
	\begin{equation}\label{zoptCondition}
	\lambda+ \beta z^{t}+y^{t} = 0.
	\end{equation}
\end{lemma}
\begin{proof}
	The claim follows from the optimality conditions of the $\bar{x}$ and $z$ updates. \qed
\end{proof}

\begin{lemma}\label{lemma: descent sequence}
	Suppose Assumptions \ref{assumption:compact}-\ref{assumption: descentOverFirstBlock} hold, and we set $\rho =2\beta$, then 
	\begin{equation}\label{eq: descent}
		L_\rho({x}^{t-1},\bar{x}^{t-1},z^{t-1},y^{t-1})- L_\rho({x}^{t},\bar{x}^{t},z^{t},y^{t})\geq \beta\|B\bar{x}^{t-1}-B\bar{x}^{t}\|^2+{\beta}\|z^{t-1}-z^{t}\|^2
	\end{equation}
	for all $t\in \Z_{++}$; in addition, there exists $\underline{L}\in \R$ independent of $k$ such that for all  $t\in \Z_{+}$,
	\begin{equation}
	L_\rho({x}^{t},\bar{x}^{t},z^{t},y^{t})\geq \underline{L} > -\infty.
	\end{equation}
\end{lemma}
\begin{proof}
	We firstly show descent over ${x}$ and $\bar{x}$ updates. By Assumption \ref{assumption: descentOverFirstBlock}, we have 
	\begin{equation}\label{eq: descent1}
		L_\rho({x}^{t-1},\bar{x}^{t-1},z^{t-1},y^{t-1})\geq L_\rho({x}^t,\bar{x}^{t-1},z^{t-1},y^{t-1}).
	\end{equation}
	In addition, notice that 
	\begin{align}\label{eq: descent2}
		&L_\rho({x}^t,\bar{x}^{t-1},z^{t-1},y^{t-1})- L_\rho({x}^t,\bar{x}^t,z^{t-1},y^{t-1})\notag\\
		=& \langle y^{t-1}, B\bar{x}^{t-1}-B\bar{x}^t\rangle+ \frac{\rho}{2}\|A{x}^t+ B\bar{x}^{t-1}+z^{t-1}\|^2- \frac{\rho}{2}\|A{x}^t+ B\bar{x}^t+z^{t-1}\|^2\notag\\
		=& \langle B^\top y^{t-1}+ \rho B^\top(A{x}^t+B\bar{x}^t+z^{t-1}),\bar{x}^{t-1}-\bar{x}^t\rangle+\frac{\rho}{2}\|B\bar{x}^{t-1}-B\bar{x}^t\|^2\notag\\ 
		\geq & \frac{\rho}{2}\|B\bar{x}^{t-1}-B\bar{x}^t\|^2,
	\end{align} 
	the second equality is due to $\|a+b\|^2-\|a+c\|^2=2(a+c)^\top (b-c)+\|b-c\|^2$ with $a = A{x}^t+z^{t-1}$, $b = B\bar{x}^{t-1}$, and $c = B\bar{x}^t$, and the last inequality is due to \eqref{x_2optCondition} of Lemma \ref{lemma:closedFormSol}. Now we will show descent over $z$ and $y$ updates. Notice that if we define $h(z) = \lambda^\top z +\frac{\beta}{2}\|z\|^2$, then by Lemma \ref{lemma:closedFormSol}, we have $\nabla h(z^t)=\lambda+\beta z^t = -y^t$; since $h(\cdot)$ is convex, it follows $h(z^{t-1})- h(z^t)+(y^t)^\top (z^{t-1}-z^t)\geq 0$. Notice that
	\begin{align}\label{eq: descent3}
		&L_\rho({x}^t,\bar{x}^t,z^{t-1},y^{t-1})-L_\rho({x}^t,\bar{x}^t,z^t,y^t)\notag\\
		=& h(z^{t-1})- h(z^t)+(y^t)^\top (z^{t-1}-z^t)+\frac{\rho}{2}\|z^{t-1}-z^t\|^2-\rho\|A{x}^t+B\bar{x}^t+z^t\|^2	\notag\\
		\geq&(\frac{\rho {+\beta}}{2}-\frac{\beta^2}{\rho})\|z^{t-1}-z^t\|^2.
	\end{align}
	The equality is due to \rev{the update of dual variable in Algorithm \ref{alg: inner}, the optimality condition \eqref{zoptCondition}, and the fact that}
	$-\rho(a+b)^\top(a+c) + \frac{\rho}{2}\|a+c\|^2-\frac{\rho}{2}\|a+b\|^2=\frac{\rho}{2}\|c-b\|^2-\rho\|a+b\|^2$ with $a =A{x}^t+B\bar{x}^t$, $b = z^t$, and $c = z^{t-1}$; the inequality is due to $h(z)$ being {$\beta$-strongly} convex and \eqref{zoptCondition} of Lemma \ref{lemma:closedFormSol}. Since $\rho=2\beta$, adding \eqref{eq: descent1}-\eqref{eq: descent3} proves \eqref{eq: descent}.
	
	To see $L_\rho({x}^{t},\bar{x}^{t},z^{t},y^{t})$ is bounded from below, we note that the function $h(z)$ defined above is also Lipschitz differentiable with constant $\beta$, so define $s^{t} := - (A{x}^t+B\bar{x}^t)$, we have
	$h(z^t)-(y^t)^\top(s^t-z^t)\ge h(s^t)-\frac{\beta}{2}\|s^t-z^t\|^2$. As a result, for all $t\in \Z_{+}$, 
	\begin{align}\label{eq: upperbounded}
	L_{\rho}({x}^t,\bar{x}^t,z^t,y^t) =& f({x}^t)+h(z^t)+(y^t)^\top(A{x}^t+B\bar{x}^t+z^t)+\frac{\rho}{2}\|A{x}^t+B\bar{x}^t+z^t\|^2\notag\notag\\
		\geq&f({x}^t)+h(s^t)-\frac{\beta}{2}\|s^t-z^t\|^2+\frac{\rho}{2}\|A{x}^t+B\bar{x}^t+z^t\|^2 \notag \\
		\geq & f(x^t) + h(s^t) \geq  f({x}^t) - \frac{\|\lambda\|^2}{2\beta}, 
	\end{align}
	where the last inequality is due to $h(s^t) = \frac{\beta}{2}\|s^t+\frac{\lambda}{\beta}\|^2-\frac{\|\lambda\|^2}{2\beta}$. Since $\lambda$ is bounded, there exists $M\in \R$ such that $\|\lambda\|^2 \leq M$; since the outer-level penalty $\beta^k$ is nondecreasing, we can define $\underline{L} := f^* -M/\beta^1,$
	where $f^* = \min_{{x}\in \mathcal{X}} f({x})$. The minimum is achievable due to Assumption \ref{assumption:compact}. \qed
\end{proof}

\rev{Now we are ready to prove Proposition \ref{prop: innerAKKT}.
\begin{proof}
	By Lemma \ref{lemma: descent sequence}, for any $T\in \Z_{++}$ we have
	\begin{align*}
		\beta \sum_{t=1}^{T} \|B\bar{x}^{t-1}-B\bar{x}^{t}\|^2+\|z^{t-1}-z^{t}\|^2 \leq \overline{L}_k - \underline{L},
	\end{align*}
	which implies the existence of a particular index $t \in [T]$ such that 
	\begin{align}\label{eq: admm_telescope}
			\|B\bar{x}^{t-1}-B\bar{x}^{t}\|^2+\|z^{t-1}-z^{t}\|^2 \leq \frac{\overline{L}_k - \underline{L}}{\beta T}.
	\end{align}
	Using the fact that $\|A{x}^t+B\bar{x}^t+z^t\| = \frac{\beta}{\rho}\|z^{t-1}-z^{t}\| = \frac{1}{2}\|z^{t-1}-z^{t}\|$, the KKT errors can be bounded by
	\begin{align*}
		& \max \{ \|\rho A^\top (B\bar{x}^{t-1}+z^{t-1}-B\bar{x}^t - z^t)\|, \|\rho B^\top (z^{t-1}- z^t)\|,\|A{x}^t+B\bar{x}^t+z^t\|\}	\\
		\leq & \rho \max\{\|A\|, \|B\|, 1/(2\rho)\} \left(\|B\bar{x}^{t-1}-B\bar{x}^{t}\| + \|z^{t-1}-z^{t}\| \right) \\
		\leq & 2\sqrt{2}\beta \max\{\|A\|, \|B\|, 1\}\left(\|B\bar{x}^{t-1}-B\bar{x}^{t}\|^2 + \|z^{t-1}-z^{t}\|^2 \right)^{1/2} \\
		\leq & 2\sqrt{2}\beta \max\{\|A\|, \|B\|, 1\} \left( \frac{\overline{L}_k - \underline{L}}{\beta T}\right)^{1/2} \leq \min \{\epsilon^k_1,\epsilon^k_2, \epsilon^k_3 \},
	\end{align*}
	where the first inequality is due to the triangle inequality, the second inequality is due to the Cauchy–Schwarz inequality and $\rho = \rho^k = 2\beta^k \geq 2\beta^0 \geq 1/2$, the third inequality is due to \eqref{eq: admm_telescope}, and the last inequality is due to the claimed upper bound on $T$. \qed
\end{proof}}

\subsection{Proof of Theorem \ref{thm: feasible}}\label{sec: proof of thm1}
\begin{proof}
	Since ${x}^k\in \mathcal{X}$, $\bar{x}^k\in \bar{\mathcal{X}}$ and $\mathcal{X}$, $\bar{\mathcal{X}}$ are bounded, we know $\|A{x}^k+B\bar{x}^k\|$ is bounded; since $\|A{x}^k+B\bar{x}^k+z^k\|\leq \epsilon^k_3$ and $\epsilon^k_3\rightarrow 0$, $\{z^k\}$ is also bounded. We conclude that $\{({x}^k, \bar{x}^k, z^k)\}$ is bounded and therefore has at least one limit point, denoted by $({x}^*, \bar{x}^*, z^*)$. We use $k_r$ to denote a subsequence converging to $({x}^*, \bar{x}^*, z^*)$. Since $\mathcal{X}$, $\bar{\mathcal{X}}$ are also closed, we have ${x}^*\in \mathcal{X}$ and $\bar{x}^*\in \bar{\mathcal{X}}$. Moreover, $A{x}^*+B\bar{x}^*+z^* = 
	\lim_{r\rightarrow \infty} A{x}^{k_r}+B\bar{x}^{k_r}+z^{k_r} = 0$. Therefore $({x}^*, \bar{x}^*)$ is feasible for problem \eqref{formulation: generic} if and only if $z^*=0$. If $\beta^k$ is bounded, then according to the update scheme, we have $z^k\rightarrow 0$, so $z^*=0$. Now suppose $\beta^k$ is unbounded. Since $\beta^k$ is nondecreasing, any subsequence is also unbounded. By \eqref{eq: appStationary3}, we have
	\begin{equation}\label{eq: z limit}
	\frac{\lambda^{k_r}}{\beta^{k_r}} + z^{k_r} + \frac{y^{k_r}}{\beta^{k_r}} = 0.
	\end{equation}
	Since $\{\lambda^{k_r}\}$ is bounded, we may assume $\lambda^{k_r}\rightarrow \lambda^*$. Again we consider two cases. In the first case, suppose $\{y^{k_r}\}$ has a bounded subsequence, and therefore has a limit point $y^*$. Then taking limit on both sides of \eqref{eq: z limit} along the  subsequence converging to $y^*$, we have $z^*=0$, so $({x}^*, \bar{x}^*)$ is feasible. Otherwise in the second case,  $\lim_{r\rightarrow\infty }\|y^{k_r}\|=+\infty$. Denote $\tilde{y}^{k_r}:=\frac{y^{k_r}}{\beta^{k_r}}$. We know the sequence $\{\tilde{y}^{k_r}\}$ converges to $-z^*$, because 
	\begin{equation}
	\lim_{r\to\infty}\tilde{y}^{k_r}=\lim_{r\rightarrow\infty} \frac{y^{k_r}}{\beta^{k_r}} = \lim_{r\rightarrow\infty} -z^{k_r} - \frac{\lambda^{k_r}}{\beta^{k_r}} = -z^*.
	\end{equation}
	By \eqref{eq: appStationary1} and \eqref{eq: appStationary2}, we have 
	\begin{align*}
	d^{k_r}_1 - \nabla f({x}^{k_r}) - A^\top y^{k_r} \in N_{\mathcal{X}}({x}^{k_r}), \quad d^{k_r}_2 - B^\top y^{k_r} \in N_{\bar{\mathcal{X}}}(\bar{x}^{k_r}).
	\end{align*}
	Since \rev{$N_{\mathcal{X}}({x}^{k_r})$ and $N_{\bar{\mathcal{X}}}(\bar{x}^{k_r})$ are cones and $\beta^{k_r}>0$}, we have 
	\begin{align}\label{eq: cone}
	&\frac{d^{k_r}_1}{\beta^{k_r}} - \frac{\nabla f({x}^{k_r})}{\beta^{k_r}} - A^\top \tilde{y}^{k_r} \in N_{\mathcal{X}}({x}^{k_r}),\quad \frac{d^{k_r}_2}{\beta^{k_r}} - B^\top \tilde{y}^{k_r} \in N_{\bar{\mathcal{X}}}(\bar{x}^{k_r}),
	\end{align}
	where $\tilde{y}^{k_r}:= \frac{y^{k_r}}{\beta^{k_r}}$. \rev{Due to the closedness of normal cones, we can take limit}  on \eqref{eq: cone}, then \eqref{eq: z limit} and \eqref{eq: appStationary4} implies $({x}^*, \bar{x}^*)$ is a stationary point of the problem \eqref{formulation: feasibility}. \qed
\end{proof}	

\subsection{Proof of Theorem \ref{thm: convergence to stationary point}} \label{sec: proof of thm2}
\begin{proof}
	We assume the subsequence $\{({x}^{k_r}, \bar{x}^{k_r}, z^{k_r}, y^{k_r})\}$ converges to the limit point $({x}^*$, $\bar{x}^*$, $z^*$, $y^*)$. Using a similar argument in the proof of Theorem \ref{thm: feasible}, we have ${x}^*\in \mathcal{X}$, $\bar{x}^* \in \bar{\mathcal{X}}$, and $A{x}^*+B\bar{x}^* + z^*= 0$. It remains to show $z^*=0$ to complete primal feasibility. If $\beta^k$ is bounded, then we have $z^k\rightarrow 0$ so $z^*=0$; if $\beta^k$ is unbounded, by taking limits on both sides of \eqref{eq: z limit}, we also have $z^*=0$, since $\lambda^k$ is bounded and $y^{k_r}$ converges to $y^*$. Therefore $({x}^*, \bar{x}^*)$ satisfies \eqref{condition: stationary3}. Taking limits on \eqref{eq: appStationary1} and \eqref{eq: appStationary2} as $k\rightarrow \infty$, we get \eqref{condition: stationary1} and \eqref{condition: stationary2}, respectively. This completes the proof. \qed
\end{proof}

\subsection{Proof of Theorem \ref{thm: rate}} \label{sec: proof of thm3}
\begin{proof}
	We use $k$ to index outer-level iterations of Algorithm \ref{alg: outer_mod} and $t$ to index inner-level iterations of Algorithm \ref{alg: inner}. 
	\rev{By Proposition \ref{prop: innerAKKT}, Assumption \ref{assumption: boundedLagrangian}, and the fact that $\beta^k = \beta^0 \gamma^k$, the number of iterations $T_k$ of the $k$-th inner ADMM, defined in \eqref{eq:Tk}, satisfies }
   \begin{equation}\label{eq: inner_complexity}
   	T_k = \mathcal{O}\left(\frac{\beta^k}{\epsilon^2}\right)=\mathcal{O}\left(\frac{\gamma^{k}}{\epsilon^2}\right).
   \end{equation}
   Summing $T_k$ over $k\in [K]$, we obtain the following bound on the total number of ADMM iterations:
   \begin{equation}\label{eq: total_complexity}
   		\sum_{k=1}^K T_k = \mathcal{O}\left( \frac{1}{\epsilon^2} \frac{ \gamma (\gamma^K -1)}{\gamma -1}\right)=\mathcal{O}\left(\frac{\gamma^K}{\epsilon^2}\right). 
   \end{equation}
   Since conditions \eqref{eq: dual_orig_1} and \eqref{eq: dual_orig_2} are maintained at the termination of each inner-level ADMM, the total number of outer-level ALM iterations, $K$,  depends on the rate at which \eqref{eq: primal_orig} is satisfied. By inequality \eqref{eq: upperbounded} and Assumption \ref{assumption: boundedLagrangian}, at the termination of each ADMM, we have
	\begin{align}\label{eq: unbounded_dual}
		\overline{L}  \geq L_{\rho^k}(x^0, \bar{x}^0,z^0, y^0) \geq f(x^k) - \langle \lambda^k, Ax^k+B\bar{x}^k\rangle +\frac{\beta^k}{2}\|Ax^k+B\bar{x}^k\|^2.
	\end{align}
	The Assumption \ref{assumption:compact}, the fact that $\|\lambda^k\|$ is bounded, and the above inequality imply that 
	$$\|Ax^k+B\bar{x}^k\|^2 = \mathcal{O}\left(\frac{1}{\beta^k}\right) = \mathcal{O}\left(\frac{1}{\gamma^k}\right).$$
	\rev{As a result, there exists an index $K$ such that $\|Ax^{K}+B\bar{x}^{K}\|\leq \epsilon$ and $\gamma^{K} = \mathcal{O}\left(1/\epsilon^2\right)$. 
	Plugging $\gamma^{K} = \mathcal{O}\left(1/\epsilon^2\right)$ into \eqref{eq: total_complexity} gives the claimed $\mathcal{O}(1/\epsilon^{4})$ complexity upper bound.}
	
	\rev{
	For the second claim, consider the $K$-th inner ADMM, at the termination of which we have $\|Ax^K+B\bar{x}^K+z^K\|\leq \frac{\epsilon}{2}$. Since $\|Ax^K+B\bar{x}^K\| \leq \|Ax^K+B\bar{x}^K+z^K\| + \|z^K\| \leq \frac{\epsilon}{2} +\|z^K\|$. It suffices to find an index $K$ such that $\|z^K\|\leq \frac{\epsilon}{2}$. Since $\hat{\lambda}^k$ and $\lambda^k$ are bounded, we have $\|z^k\| = \|\hat{\lambda}^k-\lambda^k\|/\beta^k=\mathcal{O}\left(1/\gamma^k\right)$. As a result, we can choose $K$ such that $\gamma^K = \mathcal{O}(1/\epsilon)$. Plugging $\gamma^{K} = \mathcal{O}(1/\epsilon)$ into \eqref{eq: total_complexity} gives the claimed $\mathcal{O}(1/\epsilon^{3})$ complexity upper bound. \qed}
 \end{proof}

\subsection{Proof of Theorem \ref{thm: multblock_rate}}\label{sec: proof of thm4}
\begin{proof}
According to \cite[Theorem 4.2]{jiang2019structured}, given the inner ADMM penalty $\rho^k$, which is a constant multiple of $\beta^k$, it is sufficient to let the $k$-th ADMM run $T_k = \mathcal{O} ((\rho^k)^2/\epsilon^2)$ iterations in order to have some  $t\in [T_k]$ such that the primal and dual residuals of ADMM at iteration $t$ are less than \rev{$\epsilon/2$}. Denote this solution by $\rev{x^k}=(x_1^k,\cdots, x_p^k)$. Since we update penalties in each outer iteration as $\beta^k = \beta^0 \gamma^{k}$ , the total number of inner-level iterations is bounded by
\begin{align}\label{eq: multiblock_rate}
	\sum_{k=1}^K T_k = \mathcal{O}\left(\sum_{k=1}^K\frac{(\rho^k)^2}{\epsilon^2}\right) =  \mathcal{O}\left(\frac{1}{\epsilon^2}\frac{\gamma^2 (\gamma^{2K}-1)}{\gamma^2 -1}\right)= \mathcal{O}\left(\frac{\gamma^{2K}}{\epsilon^2}\right),
	\end{align}
	where $K$ is the total number of outer-level iterations. It remains to choose $K$ such that $\|Ax^K-b\|\leq \epsilon$, and we consider two cases. 
\begin{enumerate}
	\item Suppose the ``true" dual variable $\hat{\lambda}^k = \lambda^k + \beta^kz^k$ stays bounded. It immediately follows that $\|z^k\| = \mathcal{O}(1/\beta^k)$. \rev{To get $\|z^K\|\leq \frac{\epsilon}{2}$ so that $\|Ax^k-b\|\leq \|Ax^k+z^k-b\| + \|z^k\|\leq \epsilon$,} it suffices to choose some $K$ with $\beta^K = \mathcal{O}(1/\epsilon)$, which follows $\gamma^K = \mathcal{O}(1/\epsilon)$.
	
	\item Otherwise, similar as in Theorem \ref{thm: rate}, since there is a uniform upper bound on the values of augmented Lagrangians, it suffices to let $\beta^K = \mathcal{O}(1/\epsilon^2)$, which follows $\gamma^K = \mathcal{O}(1/\epsilon^2)$.
\end{enumerate}
Finally, plugging $\gamma^K = \mathcal{O}(1/\epsilon)$ and $\gamma^K = \mathcal{O}(1/\epsilon^2)$ into \eqref{eq: multiblock_rate} will give $\mathcal{O}(1/\epsilon^4)$ and $\mathcal{O}(1/\epsilon^6)$ respectively. This completes the proof. \qed
\end{proof}

\section{Additional Proofs in Section \ref{section: Local}}
\subsection{Proof of Proposition \ref{prop: existence of local sol}} \label{sec: proof of prop3}
\begin{proof}
	Denote $L= \nabla^2 f(x^*)+\sum_{i=1}^p \mu_i^*\nabla^2 h_i(x^*)$. We firstly show under Assumption \ref{assumption: second_order_sufficient}, there exists $\underline{\beta}>0$ such that for all $\beta \geq \underline{\beta}$, we have $u^\top L u + \frac{\beta}{2}\|Au+Bv\|^2 > 0$ for all $(u,v) \neq 0$ and $\nabla h(x^*)^\top u = 0$. Suppose for any $k\in \Z_{++}$, there exists $(u^k, v^k)$ on the unit sphere such that $\nabla h(x^*)^\top u^k = 0$ and $(u^k)^\top L u^k + \frac{k}{2}\|Au^k+Bv^k\|^2 \leq 0$. Without loss of generality, assume $(u^k, v^k)$ converges to some $(\bar{u}, \bar{v})$, which is located on the unit sphere as well. Then we have 
		$\bar{u}^\top L\bar{u} + \limsup_{k\rightarrow \infty} \frac{k}{2}\|Au^k+Bv^k\|^2 \leq 0,$
	and it follows $A\bar{u}+B\bar{v} = 0$ and $\bar{u}^\top L\bar{u}\leq 0$, which is a desired contradiction since $\nabla h(x^*)^\top \bar{u} = 0$ and $(\bar{u}, \bar{v})\neq 0$.
			
	Since $\bar{x}^*\in \mathrm{Int} ~\bar{\mathcal{X}}$, we temporarily ignore the constraint $\bar{x}\in \bar{\mathcal{X}}$ and consider the system in variables $(x,\bar{x}, \tilde{\lambda}, \mu, t, \gamma, \tilde{d}_1,\tilde{d_2})$:
	\begin{alignat*}{2}
		\nabla f(x) -A^\top \tilde{\lambda} + \nabla h(x)\mu = \tilde{d}_1, &\quad  -B^\top \tilde{\lambda} = \tilde{d_2},&&\\
		-Ax-B\bar{x} + t + \gamma \lambda^* - \gamma\tilde{\lambda} = 0,&\quad h(x) = 0,&& 
	\end{alignat*}
	which has a solution $(x, \bar{x},  \tilde{\lambda},\mu) =(x^*, \bar{x}^*, \lambda^*, \mu^*)$ for $(t,\tilde{d}_1,\tilde{d}_2)= (0,0,0)$ and any $\gamma \in \R$. We claim that for any $\gamma \in [0, 1/\underline{\beta}]$, the Jacobian of the above system evaluated at $(x^*, \bar{x}^*, \lambda^*, \mu^*, 0, \gamma, 0,0)$ with respect to  $(x, \bar{x},  \tilde{\lambda},\mu)$, namely, the matrix
	\begin{align}\label{eq: matrix}
	\begin{bmatrix}
		L & 0 & -A^\top & \nabla h(x^*) \\ 
		0 & 0 & -B^\top & 0 \\
		-A & -B & -\gamma I & 0 \\
		\nabla h(x^*)^\top & 0 & 0 & 0
	\end{bmatrix},
	\end{align}
	is invertible. To see this, consider the linear system in  $(u,v, w,z)$ of proper dimensions, 
	\begin{subequations}
	\begin{align}
		L u - A^\top w + \nabla h(x^*)z &= 0, \label{eq: li_1} \\ 
		-B^\top w&= 0, \label{eq: li_2}\\
		Au+Bv+\gamma w &= 0, \label{eq: li_3} \\
		\nabla h(x^*)^\top u &= 0. \label{eq: li_4}
	\end{align}
	\end{subequations}
	For $\gamma >0$, notice that $u^\top$\eqref{eq: li_1} + $v^\top \eqref{eq: li_2}$, together with \eqref{eq: li_3} and \eqref{eq: li_4}, yields $u^\top L u+\frac{1}{\gamma}\|Au+Bv\|^2 =0$. By the first claim we know $(u,v)=0$; thus, $w=0$ by \eqref{eq: li_3}, and $z=0$ by \eqref{eq: li_1} and the fact that $\nabla h(x^*)$ has full column rank. For $\gamma =0$, using the same technique as above and \eqref{assumption: sosc}, we can show $(u,v)=0$; since we also assume gradients of all equality constraints are linearly independent, we have $(w,z)=0$ as well.
	
	 Now the Implicit Function Theorem \cite[Chapter 1.2]{bertsekas2014constrained}, together with a change of variable with $t = (\lambda-\lambda^*)/\beta$ and $\gamma = 1/\beta$, proves the existence and uniqueness of the continuous differentiable mappings $x(\cdot)$, $\bar{x}(\cdot)$, $\mu(\cdot)$, and $\tilde{\lambda}(\cdot)$ over $S$ as well as \eqref{eq: implicit_1}-\eqref{eq: implicit_2}; in addition, the $\delta$ defining $S$ can be chosen small enough so that \eqref{eq: implicit_3} holds. Finally, \eqref{eq: implicit_4} follows from the Mean Value Theorem for Integrals \cite[Proposition 2.14]{bertsekas2014constrained}. \qed
\end{proof}

\subsection{Proof of Proposition \ref{proposition: contraction}} \label{sec: proof of prop4}

\begin{proof}
	Notice that 
	\begin{align}
	& \beta^k \|Ax^k+B\bar{x}^k\| = \|\tilde{\lambda}(s^k) - \lambda^k\| \leq  \|\tilde{\lambda}(s^k) -\lambda^*\| + \|\lambda^k-\lambda^*\|\notag \\
	\stackrel{\eqref{eq: implicit_4}}{\leq} &M(\|\lambda^k - \lambda^*\|^2/(\beta^k)^2 + \|\tilde{d}^k_1\|^2+\|\tilde{d}^k_2\|^2)^{1/2} +\|\lambda^k-\lambda^*\|\notag \\ 
	\leq &\frac{M+\beta^k}{\beta^k} \|\lambda^k-\lambda^*\| + M(\|d_1^k\|+\beta^k\|A\|\|d_3^k\|) + M(\|d_2^k\|+\beta^k\|B\|\|d_3^k\|) \notag \\
	\stackrel{\eqref{eq: ub on residual}}{\leq} & \frac{M+\beta^k}{\beta^k}\|\lambda^k-\lambda^*\| + M\eta \|Ax^k+B\bar{x}^k\|, \notag 
	\end{align}
	which implies for $\beta^k > M\eta$,
	\begin{equation} \label{eq: dual contraction 1}
		\|Ax^k+B\bar{x}^k\| \leq \frac{M+\beta^k}{\beta^k(\beta^k-M\eta)}\|\lambda^k-\lambda^*\|.
	\end{equation}
	Similarly, we have 
	\begin{align}
		&\|\hat{\lambda}^k -\lambda^*\| \leq  \|\tilde{\lambda}(s^k) - \lambda^*\| + \|\hat{\lambda}^k - \tilde{\lambda}(s^k)\| \notag \\
		\stackrel{\eqref{eq: implicit_4}}{\leq} & \frac{M}{\beta^k}\|\lambda^k-\lambda^*\| + M(\|d_1^k\|+\beta^k\|A\|\|d_3^k\|) + M(\|d_2^k\|+\beta^k\|B\|\|d_3^k\|) +\beta^k\|d_3^k\| \notag \\
		\stackrel{\eqref{eq: ub on residual}}{\leq} & \frac{M}{\beta^k}\|\lambda^k-\lambda^*\| + M\eta \|Ax^k+B\bar{x}^k\| \stackrel{\eqref{eq: dual contraction 1}}{\leq}  \left(\frac{M}{\beta^k} +\frac{M\eta (M+\beta^k)}{\beta^k(\beta^k-M\eta)} \right) \|\lambda^k -\lambda^*\|. \notag
	\end{align}
	This completes the proof. \qed 
\end{proof}
\end{document}